\documentclass[12pt, reqno]{amsart}
\usepackage{amsmath}

\usepackage{algorithm}
\usepackage[noend]{algpseudocode}

\usepackage{graphicx}
\usepackage{subcaption}
\usepackage{hyperref}

\usepackage{xcolor}
\usepackage{enumitem}
\allowdisplaybreaks

\newtheorem{theorem}{Theorem}[section]
\newtheorem{lemma}[theorem]{Lemma}

\newtheorem{assumption}{Assumption}[section]
\theoremstyle{definition}

\theoremstyle{remark}
\newtheorem{remark}[theorem]{Remark}
\numberwithin{equation}{section}

\newcommand{\E}{\mathbb{E}}
\newcommand{\R}{\mathbb{R}}

\numberwithin{equation}{section}

\begin{document}
\title[AEGDM] 
{An Adaptive Gradient Method with Energy and Momentum}

\author{Hailiang Liu and Xuping Tian}
\address{Iowa State University, Mathematics Department, Ames, IA 50011} \email{hliu@iastate.edu, xupingt@iastate.edu}
\begin{abstract}
We introduce a novel algorithm for gradient-based optimization of stochastic objective functions. The method may be seen as a variant of SGD with momentum equipped with an adaptive learning rate automatically adjusted by an `energy' variable. The method is simple to implement, computationally efficient, and well suited for large-scale machine learning problems. 
The method exhibits unconditional energy stability for any size of the base learning rate.  We provide a regret bound on the convergence rate under the online convex optimization framework. We also establish the energy-dependent convergence rate of the algorithm to a stationary point in the stochastic non-convex setting. In addition, a sufficient condition is provided to guarantee a positive lower threshold for the energy variable. Our experiments demonstrate that the algorithm converges fast while generalizing better than or as well as SGD with momentum in training deep neural networks, and compares also favorably to Adam.

\end{abstract}

\maketitle

\section{Introduction}\label{Intro} 
Stochastic gradient descent (SGD) \cite{RM51} is now one of the most dominant approaches for training many machine learning (ML) models including deep neural networks (DNNs) \cite{Go16}. In each iteration, SGD only performs one parameter update on a mini-batch of training examples. Hence it is simple and has been proven to be efficient, especially for tasks on large datasets \cite{Bo12, JK+18, WR+18}.
However, the variance of SGD can slow down the convergence after the first few training epochs; a decaying step size typically has to be applied, which is one of the major bottlenecks for the fast convergence of SGD \cite{Bo12, SW96}. 
In recent years, adaptive variants of SGD have emerged and shown successes for their automatic learning rate adjustment. Examples include Adagrad \cite{Du11}, Adadelta \cite{Z12}, RMSprop \cite{TH12}, and Adam \cite{KB17}; while Adam, which may be seen as a combination of RMSprop and an exponential moving average of the first moment, stands out in this family of algorithms and  stays popular on various tasks. 
 However, training with Adam or its variants typically generalizes worse than SGD with momentum (SGDM), even when the training performance is better \cite{WR+18}. This explains why SGD(M) remains as a popular alternative. 

AEGD (Adaptive gradient decent with energy) \cite{LT20} is another gradient-based optimization algorithm that outperforms 
vanilla SGD. The distinct feature of AEGD is the use of an additional energy variable, which is updated together with the solution. The resulting algorithm is unconditionally energy stable (in the sense detailed in section 2) regardless of the base learning rate. 
Moreover, the element-wise AEGD allows for different effective learning rates for different coordinates, which has been empirically verified more effective than the global AEGD, see \cite{LT20}. 
With AEGD the effective learning rate is the base learning rate multiplied by the energy term, against a transformed gradient.

With Adam-like adaptive gradient methods, adaptation is realized by the normalization in terms of the running average of the second order moment. While these algorithms have been successfully employed in several practical applications, they have also been observed to not converge in some other settings mainly due to the relative sensitivity in such adaptation. Indeed, counterexamples are provided in recent works \cite{CL19, LX19, RK18} to show that RMSprop and Adam do not converge to an optimal solution in either convex or non-convex settings. 
In contrast, the motivation in AEGD is drawn from the perspective of dynamical systems with energy dissipation \cite{LT20}. 
AEGD is unconditionally energy stable with guaranteed convergence in energy regardless of the size of the base learning rate and the shape of the objective functions. This explains why the method can have a rapid initial training process as well as good final generalization performance. 

On the other hand, it has been long known that using momentum can help accelerate gradient descent, hence speeding up the convergence of vanilla Gradient Decent (GD) \cite{P64}. 
For many application tasks, momentum can also help reduce the variance in stochastic gradients \cite{Q99, R86}. Using momentum has become a popular tectnique in order to gain convergence speed significantly \cite{Zhu18, KB17, SM+13}.

With all these observations, a natural question is:

\smallskip
{\bf Can we take the best from both AEGD and SGDM, i.e., design an algorithm that not only enjoys the unconditional energy stability as AEGD, but also features fast convergence and generalizes well as SGDM? }
\smallskip

In this paper, we answer this question affirmatively by incorporating a running sum of the transformed gradient with the element-wise AEGD, termed as AEGDM. It is shown that AEGDM converges faster than AEGD and features certain advantages over SGDM and Adam.

We highlight the main contributions of our work as follows:
\begin{itemize}
    \item We propose a novel and simple algorithm AEGDM, integrating momentum with AEGD, 
    which allows for faster convergence.
    \item  We show the unconditional energy stability of AEGDM, and provide a regret bound on the convergence rate under the online convex optimization framework.
    In the non-convex stochastic setting, we prove the energy-dependent convergence rate to a stationary point. 
    \item We investigate the behavior of energy $r_t$ both numerically and analytically. A sufficient condition to guarantee a positive lower threshold for the energy is provided.   
    
    \item We also provide thorough experiments with our proposed algorithm on training modern deep neural networks. We empirically show that AEGDM achieves a faster convergence speed than AEGD while generalizing  better than or as well as SGDM. 
    \item Our experimental results show that AEGDM achieves better generalization performance than Adam.
\end{itemize}
Regarding the theoretical results, in this work we obtain convergence results for AEGDM, in both stochastic nonconvex setting and online convex setting.   While in \cite{LT20}, convergence analysis is provided mainly in deterministic setting.

\subsection{Further related work}
{The essential idea using an energy variable in the proposed algorithm is related to the invariant energy quadratization (IEQ) approach introduced in \cite{Y16, ZWY17} to develop linear and unconditionally energy stable numerical schemes for a class of PDEs in the form of gradient flows. The scalar auxiliary variable (SAV) approach \cite{SX18} improves IEQ by using a global energy variable. In the case without spatial effects, the two formulations are the same.  \cite{LT20} is the first work to apply this methodology to optimization problems. The resulting scheme called AEGD is known to be unconditionally energy stable, and it is also the basis for the algorithm studied in this work.}

In the field of stochastic optimization, there is an extensive volume of research for designing algorithms to speed up the convergence of SGD. 
Here we review additional related work from three perspectives.  

The first type of idea  to accelerate the convergence of GD and SGD is the use of historical gradients to adapt the step size, with renowned works including \cite{Du11, TH12, Z12} and Adam \cite{KB17}. Many further advances have improved Adam, see, e.g.,  \cite{CL19, Do16, KS17, LJ+19, LH19, LX19, RK18}.
Specifically, AMSGRAD \cite{RK18} was introduced to resolve the non-convergence issue of Adam, but the analysis is restricted only to convex problems. 
AdamW \cite{LH19} applied a simple weight decay regularization to Adam, which significantly improves the performance. There are also some hybrid methods that manage to combine the advantage of Adam and SGD 
\cite{KS17, LX19}.

Another category of work trying to speed up the convergence of GD and SGD is to apply the momentum.
A simple batch GD with constant momentum such as the heavy-ball (HB) method \cite{P64} is known to enjoy the convergence rate of $O(1/t)$ for convex smooth optimization. With an adaptive momentum the Nesterov's accelerated gradient (NAG) \cite{N83, N04} has the convergence rate up to the optimal $O(1/t^2)$. Recent advances show that NAG has other advantages such as speeding up escaping saddle points \cite{JN+18}, accelerating SGD or GD in non-convex problems \cite{Zhu18, Q99}.  One can also improve generalization of SGD in training DNNs with scheduled restart techniques \cite{OC15, RD17, WN+20}.

As discussed earlier, a major bottleneck for the fast convergence of SGD lies in its variance \cite{Bo12, SW96}, a natural idea is to reduce the variance of the stochastic gradient. Different algorithms have been proposed to achieve variance reduction. Some representative works are SAGA \cite{DB+14}, SCSG \cite{LJ+17}, SVRG \cite{JZ13},  Laplacian smoothing \cite{SW+19}, and iterative averaging methods \cite{PJ92, ZC+15}.

\subsection{Organization} 
In Section 2, we review the AEGD algorithm and the properties of its update rule. 
Section 3 presents the proposed algorithm in this work and explains how momentum is incorporated into the AEGD framework. Section 4 provides a theoretical analysis of AEGDM’s energy stability and convergence estimates in both stochastic non-convex optimization, and online convex optimization settings, respectively.   
We empirically demonstrate 
the good performance of our method
for a variety of models and datasets, as shown in Section 5. Concluding remarks are given in Section 6. All technical proofs of our theoretical results are collectively presented in Appendix A-F. A comparison between some adaptive gradient methods and the proposed method is given in Appendix \ref{A-framewrok}.
Overall, we show that AEGDM is a versatile algorithm that scales to large-scale high-dimensional machine learning problems.

\subsection{Notation}  
Throughout this paper, we denote the list $\{1, \cdots, m\}$ as $[m]$ for integer $m$. For vectors and matrices we use $\|\cdot\|$ to denote the $l_2$-norm. For a function $f: \mathbb{R}^n \to \mathbb{R}$, we use $\nabla f$ and $\partial_i:= \partial_{\theta_i}$ to denote its gradient and partial derivative, respectively. For vector $\theta\in\mathbb{R}^n$, we denote its $i$-th coordinate at $t$-th iteration by $\theta_{t, i}$.
In the algorithm description we use $z=x/y$ to denote element-wise division if $x$ and $y$ are both vectors of the same size; $x\odot y$ is element-wise product.

\section{Review of AEGD}\label{sec2} 
Various gradient methods have been proposed in order to achieve better performance on diverse stochastic optimization tasks. The most common such task in machine learning is that of training feedforward neural networks. In this problem, we are given a set of labelled data points $\{x_i, y_i\}_{i=1}^m$, called the training data, and generate a  network output function $\hat f(\theta, x)$ from a feedforward neural network, where $\theta$ corresponds to a collection of the network parameters, then the loss of the network over data point $(x_i, y_i)$ is given by
$
L_i(\theta):= l(\hat f(\theta, x_i), y_i)  
$
for $\theta \in \mathbb{R}^n$. The objective function becomes \begin{equation}\label{dloss}
f(\theta)=\mathbb{E}_{x_i, y_i}[
l(\hat f(\theta, x_i), y_i)], 
\end{equation} 
which is the average loss across data points. The goal is to fit the network parameters so that to minimize the loss over the data. For most commonly-used activation and loss functions, the above function is non-convex.
When $m$ is large, SGD or its variants is preferred for solving (\ref{dloss}) mainly because of their cheapness per iteration cost. 

The key idea of SGD is to modify the updates of GD to be  
$$
\theta_{t+1}=\theta_t -\eta g_t,
$$
where $g_t$ is a stochastic estimator of the gradient with $\mathbb{E}[g_t]=\nabla f(\theta_t)$ and bounded second moment $\mathbb{E}[\|g_t\|_2^2]$. Getting $g_t$ in the posed problem  (\ref{dloss}) is simple, at each iteration $t$ we can take $g_t=\nabla L_{i_t}(\theta_t)$, where $i_t\in[m]$ is picked uniformly at random at step $t$. 

Typically, $ L_i(\theta)$ is bounded from below, that is, 
$$
L_i(\theta)>-c, \quad i \in [m],
$$
for some $c>0$, then $f(\theta):=\frac{1}{m}\sum_{i=1}^{m}L_i(\theta)>-c$. The key idea of AEGD introduced in \cite{LT20} is the use of an additional energy variable $r$ such that 
$$
\nabla f(\theta) = 2rv, \quad  
v:=\nabla \sqrt{f(\theta)+c}, 
$$
where $r$, taking as $\sqrt{f(\theta)+c}$ initially, will be updated together with $\theta$, and $v$ is a transformed gradient. The gradient flow $\dot \theta=-\nabla f(\theta)$ is then replaced by 
$$
\dot \theta=-2rv, \quad \dot r= v\cdot \dot \theta. 
$$
A simple implicit-explicit discretization gives the following AEGD update rule:
$$
\theta_{t+1}=\theta_t-2\eta r_{t+1}v_t, \quad r_{t+1}-r_t=v_t \cdot (\theta_{t+1}-\theta_t).
$$
This yields a decoupled update for $r$ as 
$r_{t+1}=r_t/(1+2\eta |v_t|^2)$. 

{
\begin{remark}
To see why $r=\sqrt{f+c}$ is a reasonable choice to develop efficient optimization algorithms, we consider a more general setting $r=(f+c)^\alpha$ where $\alpha \in (0, 1)$. Then the corresponding gradient flow becomes
$$
\dot \theta=-\alpha^{-1} r^{1/\alpha -1}v, \quad
\dot r = v \cdot \dot \theta.
$$
Using similar implicit- explicit discretization, one can see that the update for $r$ is linear if and only if $\frac{1}{\alpha}-1=1$, that is $\alpha=\frac{1}{2}$. Any other choices would make the resulting algorithm too cumbersome to use in practice.
\end{remark}
}

Presented in Algorithm \ref{alg:AEGD} is the element-wise version of the stochastic AEGD proposed and analyzed  in \cite{LT20}.
\begin{algorithm}
\caption{Stochastic AEGD. Good default setting for parameters are $c=1$ and $\eta=0.1$. }
\label{alg:AEGD}
\begin{algorithmic}[1] 
\Require  $\{L_i({\theta})\}_{i=1}^m$, $\eta$: the step size, ${\theta}_0$: initial guess of $\theta$, and $T$: the total number of iterations.
\Require $c$: a parameter such that for any $i\in[m]$, $L_i({\theta})+c>0$ for all $\theta \in \mathbb{R}^n$, initial energy: $r_0=\sqrt{L_{i_0}({\theta}_0)+c}{\bf 1}$
\For{$t=0$ to $T-1$}
\State $v_t:=\nabla L_{i_t}(\theta_{t})/\big(2\sqrt{L_{i_t}(\theta_t)+c}\big)$ 
($i_t$ is a random sample from $[m]$ at step $t$) %
\State ${r}_{t+1} = {r}_{t}/(1+2\eta v_t\odot v_t)$ (update energy)
\State ${\theta}_{t+1} = {\theta}_{t} - 2\eta {r}_{t+1}\odot v_t $
\EndFor
\State \textbf{return} ${\theta}_T$
\end{algorithmic}
\end{algorithm}
Algorithm \ref{alg:AEGD} is shown to be unconditionally energy stable in the sense that for any step size $\eta>0$,
   \begin{equation}\label{srei}
   \E[r_{t+1, i}^2]=\E[r_{t, i}^2] -\E[(r_{t+1, i} -r_{t, i})^2]- \eta^{-1} \E[(\theta_{ t+1, i}-\theta_{t, i})^2],\quad i\in [n],
   \end{equation}
that is $\E[r_{i, t}]$ is strictly decreasing and convergent with $\E[r_{t, i}] \to r_i^*$ as $t\to \infty$. For AEGD, $r_0=\sqrt{L_{i_0}(\theta_0)+c}{\bf 1}$. Here $c$ is also a hyperparameter, but it requires almost no tuning as long as  $L_{i}(\theta)+c>0$ for any $i\in [m]$ and $\theta \in \mathbb{R}^n$; see \cite{LT20}. 

The energy stability property and convergence rates obtained in \cite{LT20} apply well for arbitrary sampling of form 
\begin{equation}\label{fxi}
f(\theta;\xi)=\frac{1}{m}\sum_{j=1}^m \xi_j L_j(\theta),
\end{equation} 
where $\xi \in \mathbb{R}^m_+$ is a random sampling vector (drawn from some distribution) such that $\mathbb{E}[\xi_j]=1$ for $j\in[m].$ Of particular interest is the b-minibatch sampling: $\xi\in \mathbb{R}^m_+$ is a b-minibatch sampling if for every subset $M\subset [m]$ with $|M|=b$ so that 
$\mathbf{\xi}=\frac{m}{b}\sum_{i\in M}{\bf e}_i$.
\section{The proposed algorithm}\label{sec3} 
In this section, we propose a novel algorithm for speeding up AEGD with momentum.   
We denote the realizations of the stochastic objective function $f(\theta; \xi)$ at subsequent time steps $0,...,T$ by $f_0(\theta),...,f_T(\theta)$, then $f_t(\theta_t)$ and $\nabla f_t(\theta_t)$ are stochastic function value and gradient at step $t$, respectively. This set up for problem (\ref{fxi}) corresponds to 
\begin{equation}\label{stoc}
 f_t(\theta_t)=f(\theta_t; \xi_t).    
\end{equation}
This way the stochasticity may come from the evaluation at random samples (mini-batches) of data points. 
Such set up is more general and links to  the online optimization \cite{Ha19,Zin03}, for which one must select a point in the parameter space before seeing the cost function for that step. For the static case it reduces to   $f_t(\theta_t)= f(\theta_t)$. 

Under the current setting (online or stochastic with (\ref{stoc})), if we assume $f_t(\theta)+c >0$ for any $t\in[T]-1$ and $\theta\in\mathbb{R}^n$, the AEGD method can be reformulated as:
\begin{subequations}
\begin{align}
& v_{t,i} = \frac{\partial_i f_{t}(\theta_t)}{2\sqrt{f_{t}(\theta_t)+c}},\quad i\in[n], \\
& r_{t+1,i} =\frac{r_{t,i}}{1+2\eta v_{t,i}^2},\quad r_{0,i}=\sqrt{f_{0}(\theta_0)+c}, \\
& \theta_{t+1,i}=\theta_{t,i}-2\eta r_{t+1,i}v_{t,i}.
\end{align}
\end{subequations}
This method differs from other adaptive gradient methods in that the adaptation here is through the update of an auxiliary energy variable $r_t$. Keeping this adaptive feature, we propose a momentum update by accumulating a running sum of the historical values of $v_t$, i.e., 
$m_{t+1}=\sum_{i=0}^t \mu^i v_{t-i}$. 
Such AEGD with momentum (i.e., AEGDM) can be expressed as 
\begin{subequations}\label{aegdm}
\begin{align}
& v_{t,i} = \frac{\partial_i f_{t}(\theta_t)}{2\sqrt{f_{t}(\theta_t)+c}},\quad i\in[n], \\
& m_{t+1,i} = \mu m_{t,i}+ v_{t,i},\quad m_{0,i}=0, \\
& r_{t+1,i} =\frac{r_{t,i}}{1+2\eta v_{t,i}^2},\quad r_{0,i}=\sqrt{f_{0}(\theta_0)+c}, \\
& \theta_{t+1,i}=\theta_{t,i}-2\eta r_{t+1,i}m_{t+1,i}. 
\end{align}
\end{subequations}
The added moving sum allows for the following reformulation
$$
 \theta_{t+1,i}=\theta_{t,i}-2\eta r_{t+1,i}v_{t,i} +\mu \frac{r_{t+1, i}}{r_{t, i}} (\theta_{t, i}-\theta_{t-1, i}),
$$
which is similar to the classical momentum since $\left(\mu r_{t+1, i}/r_{t, i}\right)\leq \mu <1$. 
\begin{remark} For the momentum one may also use $m_{t+1} = \mu m_t + (1-\mu)v_{t}$ to obtain an alternative update rule. With this choice, the related bounds in our theoretical results would differ through the factor $(1-\mu)$. Also, this choice appears to bring similar performance when using the base learning rate of size $\eta/(1-\mu)$, as evidenced by our preliminary numerical tests. 
\end{remark} 
We present the procedure of AEGDM in Algorithm \ref{alg2}. {Here we want to point out that the computational complexity of AEGDM is at the same level of SGD. The additional cost only comes from the explicit update of two extra variables $m$ and $r$ in each iteration.} 
\begin{algorithm}
\caption{AEGDM. Good default setting for parameters are $c=1$, $\eta=0.01$, $\mu=0.9$}
\label{alg2}
\begin{algorithmic}[1] 
\Require A sequence of objective functions $\{f_t\}_{t=0}^{T-1}$; a constant $c$ such that $f_t(\theta)+c>0$ for all $t\in[T]-1$; base learning rate $\eta$
\Require Initialize: $\theta_0$; $m_0=0$; $r_0=\sqrt{f_0(\theta_0)+c}\,\bf{1}$
\For{$t=0$ to $T-1$}
\State $v_t=\nabla f_{t}(\theta_t)/(2\sqrt{f_{t}(\theta_t)+c})$ (transformed gradient)
\State $m_{t+1}=\mu m_{t}+ v_t$ (momentum update)
\State $r_{t+1}=r_{t}/(1+2\eta v_t\odot v_t)$ (energy update)
\State $\theta_{t+1}=\theta_t-2\eta r_{t+1}\odot m_{t+1}$
\EndFor
\State \textbf{return} ${\theta}_T$
\end{algorithmic}
\end{algorithm}

\section{Theoretical results}\label{sec4} 
In this section, we present theoretical results for both online setting and the stochastic setting, respectively. 

\subsection{Online setting}
We first present the energy stability and solution properties of AEGDM \eqref{aegdm},  and then derive a regret bound for it in the online convex setting. 
\begin{theorem}[Energy stability and solution properties]
\label{prop1} AEGDM (\ref{aegdm}) is unconditionally energy stable in the sense that for any step size $\eta>0$ and each $i\in [n]$, $r_{t, i}$ is strictly decreasing and convergent with $r_{t, i} \to r^*_i$ as $t\to \infty$. Moreover, we have the following:\\
(i) for any $\mu<1$ and $\eta>0$,  
\begin{equation}\label{rev1}
\lim_{t\to \infty} \|\theta_{t+1}-\theta_{t}\|=0, \quad 
\sum_{t=0}^\infty\|\theta_{t+1}-\theta_t\|^2 \leq \frac{2\eta n}{(1-\mu)^2}(f_0(\theta_0)+c);
\end{equation}
(ii) for any $\eta>0$, 
\begin{align}\label{T}
\frac{1}{T}\sum_{t=0}^{T-1}|v_{t,i}|
\leq 
\Bigg(\frac{\sqrt{f_0(\theta_0)+c}}{2}\Bigg)^{1/2}\frac{1}{\sqrt{\eta Tr_{T,i}}}, \quad i\in [n]. 
\end{align}

\end{theorem}
The proof is deferred to Appendix \ref{pfp1}.
\begin{remark} 
(i) In the static case, $f_t(\theta)=f(\theta)$. Even in such case, $f(\theta_t)$ may not be decreasing in $t$ unless $\eta$ is sufficiently small. However, $r_t$, which serves to approximate $\sqrt{f(\theta_t)+c}$, is strictly decreasing for any $\eta>0$. This is why such property is termed as energy stability. \\ 
(ii) The unconditional energy
stability featured by AEGDM also implies convergence of the sequence $\|\theta_{t+1}-\theta_t\|$ to zero at a rate of at least $1/\sqrt{t}$. But this is not
sufficient--at least in general--to guarantee the convergence of $\{\theta_t\}_{t\geq 0}$, unless further control on this sequence is available.\\ 
(iii) The coordinate-wise estimate in (\ref{T}) allows the control of the average of $|v_t|$ for each direction. This estimate indicates that the scheme convergence is inseparable from the asymptotic behavior of $r_t$. 
\end{remark}
Within the online learning framework proposed in \cite{Zin03},  
at each step $t$, the goal is to predict the parameter $\theta_t$ and evaluate it on a previously unknown cost function $f_t$. The nature of the sequence is unknown in advance, we evaluate our algorithm using the regret, that is the sum of all the previous difference between the online prediction $f_t(\theta_t)$ and the best fixed point parameter $f_t(\theta^*)$ from a feasible set $\Theta$: 
$$
R(T)=\sum_{t=0}^{T-1}[f_t(\theta_t)-f_t(\theta^*)], 
$$
where $\theta^* = {\rm argmin}_{\theta\in \Theta} \sum_{t=0}^{T-1}f_t(\theta)$. 
We are able to bound $R$ as stated in the following. 
\begin{theorem}[Regret guarantee]
\label{thm1} Given the sequence $\{\theta_t\}$ generated by AEGDM \eqref{aegdm} with $\mu < 1$ and $\eta>0$.
Assume that $\|\theta-\theta'\|_\infty\leq D_\infty$ for all $\theta,\theta'\in\Theta$, and $0< f_t(\theta_t)+c\leq B$
for all $t\in[T]-1$. 
When $\Theta$ and $\{f_t\}$ are convex, AEGDM achieves the following bound on the regret, for all $T\geq 1$, 
\begin{equation}\label{RT}
R(T) \leq  C_1 \Bigg(\sum_{i=1}^n \frac{1}{\eta r_{T, i}}\Bigg)^{1/2}\sqrt{T}+ C_2,
\end{equation}
where $C_1, C_2$ are constants depending on $\mu, D_\infty, B, n$ and $f_0(\theta_0)+c$. 
\end{theorem}
The proof and precise expressions for $C_1, C_2$ are deferred to Appendix \ref{pf1}.
\begin{remark} 
(i) The regret bound in \eqref{RT} 
and the bound in Theorem \ref{thm2c} may be seen as a posteriori results due to their dependence on the energy variable $r$ at the $T$-th iteration. These bounds can be useful in planning a learning rate decay schedule based on the updated $r_t$. 
If $r^*_i>0$, then $R(T)$ is of order $O(\sqrt{T})$, 
which is known the best possible bound for online convex optimization \cite[Section 3.2]{Ha19}. 
$r^*_i>0$ is shown to be conditionally true in the stochastic setting (Theorem \ref{thm2r}). We observe from  experimental results (See Figure \ref{fig:minr} in Appendix \ref{spep}) that $r_{t, i}$ decays slowly at rate of 
$ t^{-\alpha}$ with $0<\alpha<\frac{1}{2}$ (an algebraic decay rate) for large $t$ with $\eta$ within a reasonable range. In such case, the upper bound of $R(T)$ is of order $O(T^\frac{\alpha+1}{2})$. 
This still ensures the convergence in the sense that
$$
\lim_{T \to \infty} \frac{R(T)}{T}=0.
$$
It would be of interest to further investigate when this might fail.\\ (ii) The bound on $\theta_t$
is typically enforced by projection onto $\Theta$ \cite{Zin03}, with which the regret bound (\ref{RT}) can still be proven since projection is a contraction operator \cite[Chapter 3]{Ha19}.
\end{remark}

\subsection{Stochastic setting}
We proceed to present theoretical results for AEGDM \eqref{aegdm} in the stochastic setting. Our aim is at solving the following stochastic nonconvex optimization problem
$$
\min_{\theta \in \mathbb{R}^n} \left\{ f(\theta):= \mathbb{E}_\xi[f(\theta; \xi)]\right\},
$$
where $\xi$ is a random variable satisfying certain distribution, and $f(\theta; \xi): \mathbb{R}^n \to \mathbb{R}$  is a differentiable nonconvex function and bounded from below so that $f(\theta; \xi)+c>0$. In the stochastic setting, one can only get estimators of $f(\theta)$ and its gradient,  $f(\theta; \xi)$ and $\nabla f(\theta; \xi)$, respectively, with which we take $v_t=\nabla f(\theta_t; \xi_t)/(2\sqrt{f(\theta_t; \xi_t)+c})$ in Algorithm \ref{alg2}. 

In the stochastic setting, unconditional energy stability and solution properties in Theorem \ref{prop1} may be stated in the following. 

\begin{theorem}[Energy stability and solution properties]
\label{prop2} AEGDM (\ref{aegdm}) is unconditionally energy stable in the sense that for any step size $\eta>0$, $\E[r_{t, i}]$ is strictly decreasing and convergent with $\E[r_{t, i}] \to r^*_i$ as $t\to \infty$. Moreover, we have the following:\\
(i) for any $\mu<1$ and $\eta>0$,  
\begin{equation}\label{rev1+}
\lim_{t\to \infty} \E[\|\theta_{t+1}-\theta_{t}\|]=0, \quad 
\sum_{t=0}^\infty \E[\|\theta_{t+1}-\theta_t\|^2] \leq \frac{2\eta n}{(1-\mu)^2} (f(\theta_0)+c);
\end{equation}
(ii) for any $\eta>0$, 
\begin{align}\label{T+}
\frac{1}{T}\sum_{t=0}^{T-1}\E[|v_{t,i}|]
\leq 
\Bigg(\frac{\sqrt{f(\theta_0)+c}}{2}\Bigg)^{1/2}
\left(\E\left[\frac{1}{\eta Tr_{T,i}}\right]\right)^{1/2}, \quad i\in [n]. 
\end{align}
\end{theorem}
In order to present the convergence rate of AEGDM \eqref{aegdm}, we make assumptions that are commonly used for analyzing the convergence of a stochastic algorithm for nonconvex problems:
\begin{assumption}\label{asp}
\begin{enumerate}[label=\arabic*.]\setlength{\itemindent}{-2.2em}
\item {\bf Smoothness:} The objective function is $L$-smooth: $\forall x,y\in\R^n$, 
$$f(y)\leq f(x)+\nabla f(x)^\top(y-x)+\frac{L}{2}\|y-x\|^2.$$
\item {\bf Independent samples:} The random samples $\{\xi_t\}_{t=1}^\infty$ are independent.
\item {\bf Unbiasedness:} The estimate of the gradient and function value are unbiased: 
$$
\E_{\xi_t}[\nabla f(\theta_t;\xi_t)]=\nabla f(\theta_t),\quad \E_{\xi_t}[f(\theta_t;\xi_t)]= f(\theta_t).
$$
\item {\bf Bounded variance:} The variance of the estimator of both gradient and function value satisfy
$$
\E_{\xi_t}[\|\nabla f(\theta_t;\xi_t)-\nabla f(\theta_t)\|^2_2]\leq \sigma^2_g,\quad \E_{\xi_t}[|f(\theta_t;\xi_t)- f(\theta_t)|^2]\leq \sigma^2_f.
$$
\end{enumerate}
\end{assumption}

\begin{theorem}[Convergence rate]\label{thm2c} Let $\{\theta_t\}$ be the solution sequence generated by AEGDM \eqref{aegdm} with $f_t$ in the form of \eqref{stoc} and $\mu<1$, $\eta>0$. Under Assumption \ref{asp} and assume that the stochastic gradient and function value are bounded such that $\|\nabla f(\theta_t;\xi_t)\|\leq G_\infty$ and $0<a\leq f(\theta_t,\xi_t)+c\leq B$, then for all $T\geq 1$,
\begin{align*}
\frac{1}{T}\E\Bigg[\min_i r_{T,i} \sum_{t=0}^{T-1}\|\nabla f(\theta_t)\|^2_2\Bigg]
\leq \frac{C_1+C_2 n+C_3\sigma_g \sqrt{ nT}}{\eta T},
\end{align*}
where $C_1, C_2, C_3$ are constants depending on $\mu, \eta, L, G_\infty, a, B, n$ and  $f_0(\theta_0)+c$. 
\end{theorem}
The proof and the precise expressions for $C_1, C_2, C_3$ are deferred to Appendix \ref{pf2c}.

The question of how $r_{T}$ depends on $T$ is theoretically interesting but subtle to characterize. Numerically we observe that for $\eta$ less than a threshold, $r_{T, i}$ tends either to a positive number ($r^*_i>0$) or to zero much slower than $1/\sqrt{T}$ (See Figure \ref{fig:minr} in Appendix \ref{spep}). The above result is meaningful in both cases. In the case $r^*_i>0$, the rate of $O(1/\sqrt{T})$ is recovered from this bound. Next we shall identify a sufficient condition for ensuring  such lower positive threshold for $r_{T, i}$.

\subsection{Lower bound of the energy}
 First note that the $L$-smoothness of $f(\theta)$ implies the $L_F$-smoothness of $F(\theta)=\sqrt{f(\theta)+c}$ with
\begin{equation}\label{LF}
L_F=
\frac{1}{2\sqrt{f(\theta^*)+c}} \left( L+ \frac{G^2_\infty}{2(f(\theta^*)+c)}\right).     
\end{equation}
This will be used in our analysis of the asymptotic behavior of the energy. The result is stated below. 
\begin{theorem}[Lower bound of $r_T$]
\label{thm2r}
Under the same assumptions as in Theorem \ref{thm2c}, 
we have
\begin{equation}\label{minr}
\min_i\E[r_{T,i}]\geq \max\{\sqrt{f(\theta^*)+c}-\eta D_1-\mu D_2-\sigma D_3,0\}   
\end{equation}
where $\sigma=\max\{\sigma_f,\sigma_g\}$,  with $L_F$ given in (\ref{LF}) and
\begin{align*}
&D_1 = \frac{L_F n (f(\theta_0)+c)}{(1-\mu)^2}, \quad D_2 =\frac{1}{2}\bigg(1+\frac{1}{(1-\mu)^2}\bigg)n\sqrt{f(\theta_0)+c},\\
&D_3 = \frac{1}{2a^{1/2}} +\sqrt{\frac{G_\infty^2}{4a^3}+\frac{1}{a}} 
\frac{\sqrt{f(\theta_0)+c}}{1-\mu}\sqrt{\eta nT}.
\end{align*}
Moreover, in the absence of noise, we have 
\begin{equation}\label{aa}
\min_i r_{T,i}>\min_i r^*_i>0 \quad\text{if}\quad \eta D_1 +\mu D_2 <\sqrt{f(\theta^*)+c}.
\end{equation}
\end{theorem}
The derivation of $L_F$ and the proof for Theorem \ref{thm2r} are deferred to Appendix \ref{pf2r}.

\begin{remark}
1. (\ref{aa}) is only a sufficient condition. We observe from our experimental results that the upper bound for $\eta$ to guarantee the positiveness of $r^*_i$ can be much larger (See Figure \ref{fig:minr} in Appendix \ref{spep}).

2. 
In Theorem \ref{thm2r}, we measure how far $r^*$ can deviate from $F(\theta^*)$ in the worst situation. 
In the case of no momentum and no noise, we have 
$$
\min_i r_{T,i}>\min_i r^*_i>0 
\quad\text{if}\quad \eta<\frac{\sqrt{f(\theta^*)+c}}{D_1}.
$$
This recovers the result in \cite{LT20} for the deterministic AEGD. 
\end{remark}

\section{Numerical experiments} \label{sec6} 
In this section, we empirically evaluate AEGDM and compare it with SGDM, Adam, and AEGD on several benchmark problems. We show that in the deterministic case, AEGDM does speed up the convergence of AEGD; in the stochastic case like training deep neural network tasks, the fluctuation of test accuracy of AEGDM is much smaller than that of AEGD (which confirms that AEGDM reduces the variance of AEGD), while AEGDM still achieves comparable final generalization performance of AEGD, which is as well as or  better than SGDM. Adam makes rapid initial progress but does not generalize well at the end as the other three methods.


We begin with testing the deterministic counterpart of the method on the 2D-Rosenbrock function, then we conduct the stochastic version 
on several image classification tasks, including three datasets: MNIST\footnote{\url{http://yann.lecun.com/exdb/mnist/}}, CIFAR-10 \& CIFAR-100 \cite{KH09}; and six convolution neural network architectures: LeNet-5 \cite{LB98}, VGG-16 \cite{SZ15}, ResNet-32 \cite{HZ16}, DenseNet-121 \cite{HL+17}, SqueezeNet \cite{IM+16}, GoogleNet \cite{SL+15}. We choose these architectures because of their broad importance and superior performance on several benchmark tasks.

In all experiments, we set the parameter $c = 1$ for both AEGD and AEGDM. The momentum parameter $\mu$ is set to the default value $0.9$ for both SGDM and AEGDM. For Adam, we also directly apply the default hyperparameter values with $\beta_1 = 0.9$ and $\beta_2 = 0.999$. For each method, we only tune the base learning rate. Details about tuning learning rate for training deep neural networks are presented in Section \ref{img}.

\subsection{2D-Rosenbrock function}
We first compare AEGDM with AEGD and GD with momentum (GDM) on the 2D-Rosenbrock function defined by 
\begin{equation}\label{eq:rosenbrock}
  f(x_1,x_2)=(1-x_1)^2+100(x_2-x_1^2)^2.
\end{equation}
For this non-convex function, 
the global minimum $f^*=0$ achieved at $(1, 1)$ is inside a long, narrow, parabolic shaped flag valley. 
It is trivial to find the valley, but known to be difficult to converge to the global minimum. 

\begin{figure}[ht]
\begin{subfigure}[b]{0.5\linewidth}
\centering
\includegraphics[width=0.6\linewidth]{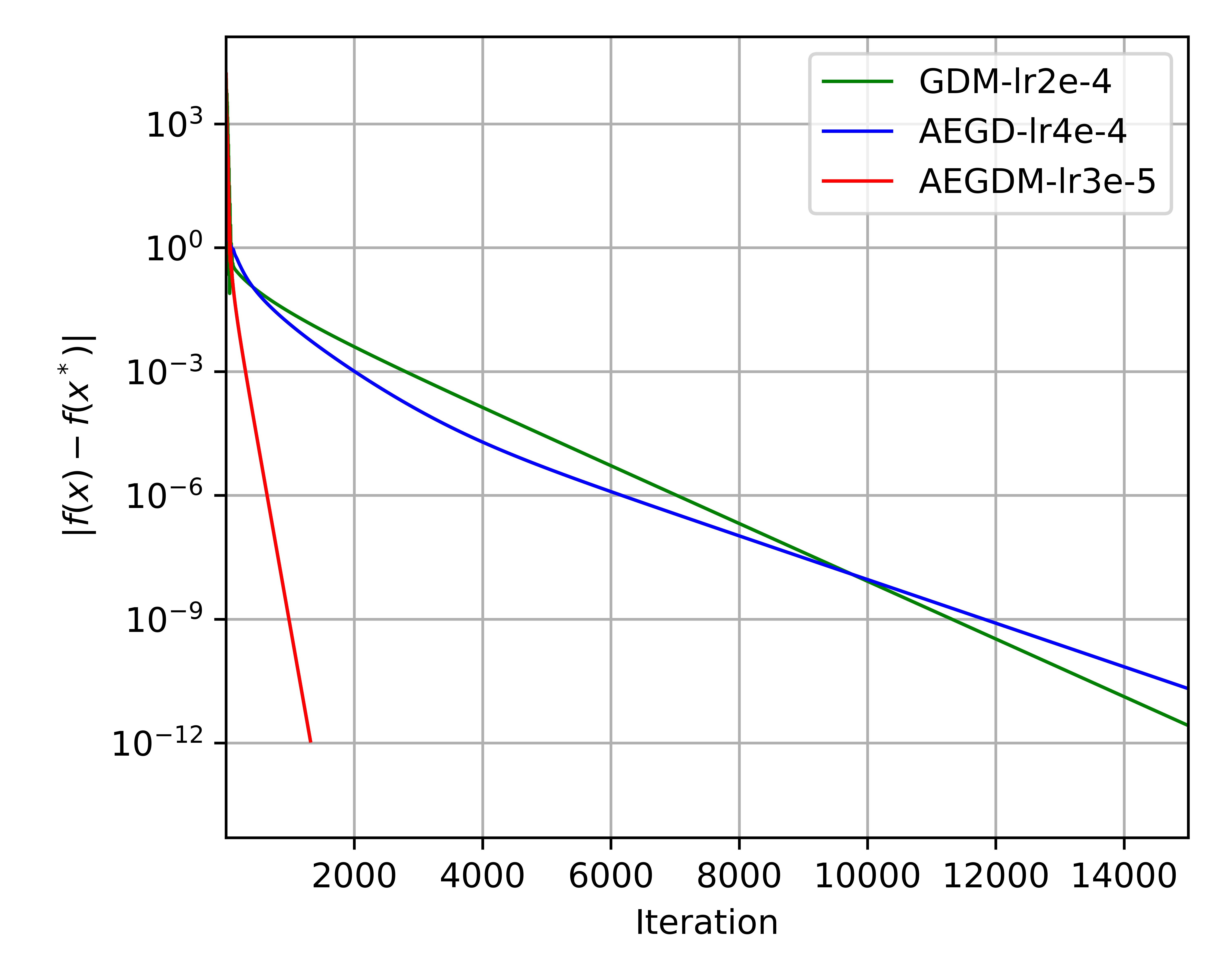}
\caption{Optimality gap}
\end{subfigure}%
\begin{subfigure}[b]{0.5\linewidth}
\centering
\includegraphics[width=0.6\linewidth]{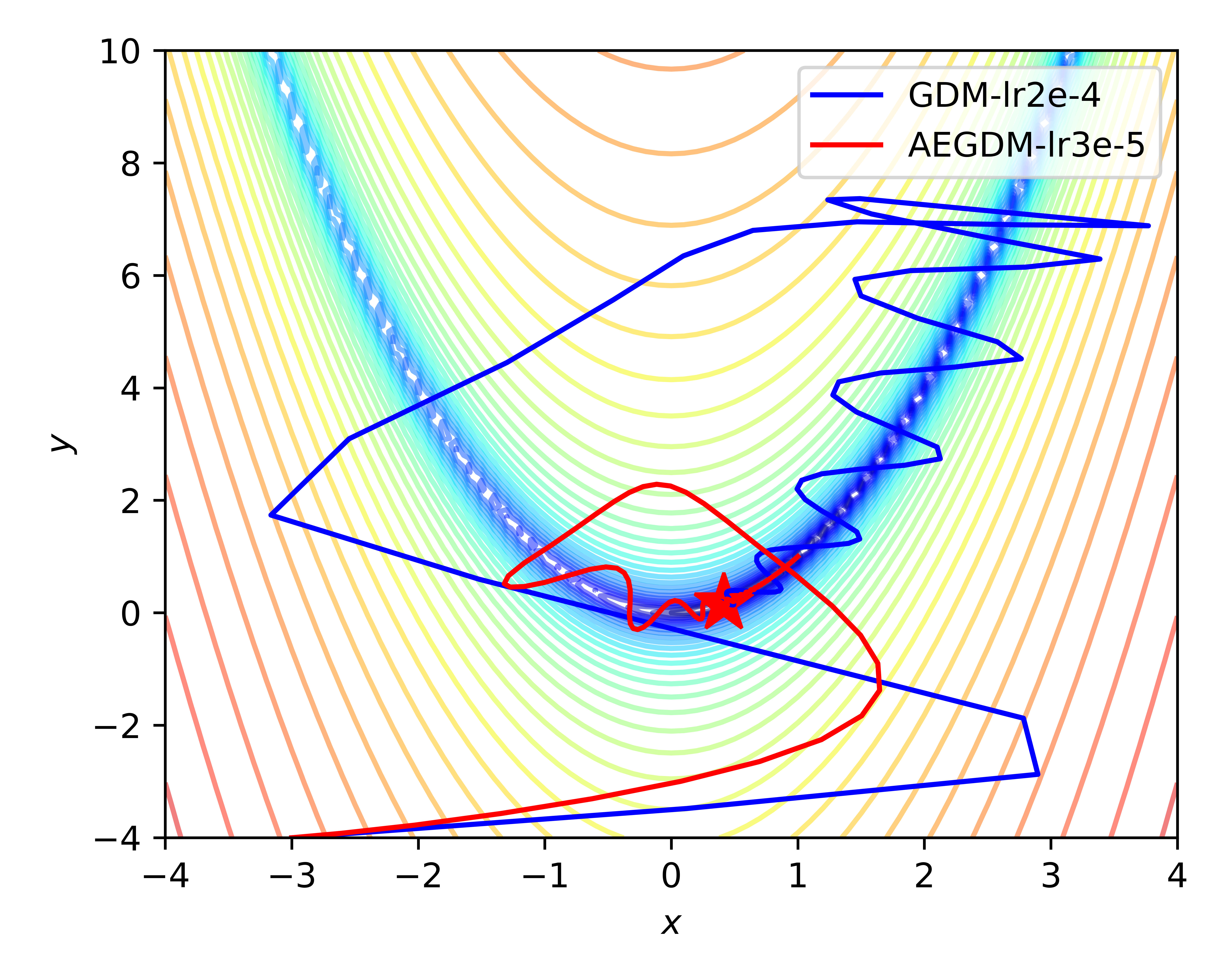}
\caption{Trajectory}
\end{subfigure}%
\caption{Comparison of the performance of three different methods on the 2D Rosenbrock function.}
\label{fig:rosen}
\end{figure}%

The initial point is set at $(-3, -4)$. We fine tune the step size and choose the one that achieves the fastest convergence speed for each method. Figure \ref{fig:rosen} (a) presents the optimality gap v.s. iteration of the three methods, the step size (represented by `lr') used for each method is also included. We see that the performance of AEGD is comparable with that of GDM, and AEGDM converges much faster than both of them. This confirms that momentum can speed up the convergence of AEGD. We also present the trajectory of GDM and AEGDM in Figure \ref{fig:rosen} (b), it can be seen that though the step size of AEGDM is a lot larger than that of GDM, the oscillation of AEGDM is much smaller, this explains why AEGDM converges much faster than GDM. 

\subsection{Image classification}\label{img}
Now we compare the performance of AEGDM with SGDM, Adam and AEGD on several image classification tasks, including LeNet-5 on MNIST; VGG-16, ResNet-32, DenseNet-121 on CIFAR-10; and SqueezeNet, GoogleNet on CIFAR-100. For experiments on MNIST, we run $50$ epochs with a minibatch size of $128$ and weight decay of $1\times 10^{-4}$. For experiments on CIFAR-10 and CIFAR-100, we employ the fixed budget of 200 epochs and reduce the learning rates by 10 after 150 epochs. Detailed settings like batch size, and weight decay were chosen as suggested for respective base architectures. We also summarize them in Appendix \ref{spep}. \footnote{We make our code available at \url{https://github.com/txping/AEGDM}.}


For each method, we fine tune the base learning rate and report the one that achieves the best final generalization performance. For SGDM, we search the base learning rate $\eta$ among $\{0.01, 0.05, 0.1, 0.2, 0.3\}$; for Adam, we search $\eta$ from $\{0.0001, 0.0003, 0.0005, 0.001, 0.002\}$; for AEGD, we search $\eta$ from $\{0.05, 0.1, 0.2, 0.3, 0.4\}$; for AEGDM, we search $\eta$ from:\\ $\{0.005, 0.008, 0.01, 0.02, 0.03\}$. We found that these choices work well for a wide range of tasks, including examples given below.

\begin{figure}[ht]
\begin{subfigure}[b]{0.33\linewidth}
\centering
\includegraphics[width=1\linewidth]{plots/cifar10_vgg16_train.png}
\caption{VGG-16, training loss}
\end{subfigure}%
\begin{subfigure}[b]{0.33\linewidth}
\centering
\includegraphics[width=1\linewidth]{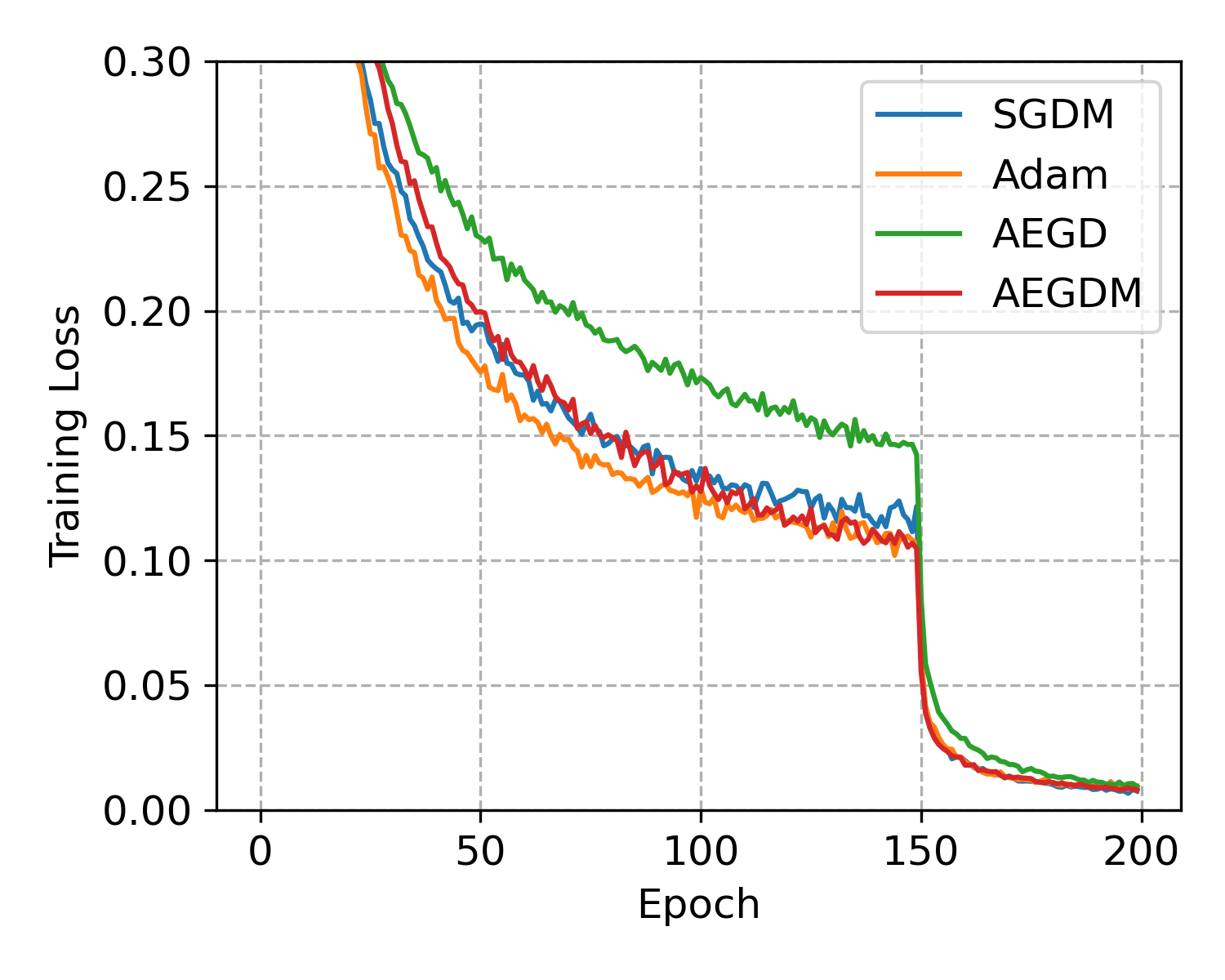}
\caption{ResNet-32, training loss}
\end{subfigure}%
\begin{subfigure}[b]{0.33\linewidth}
\centering
\includegraphics[width=1\linewidth]{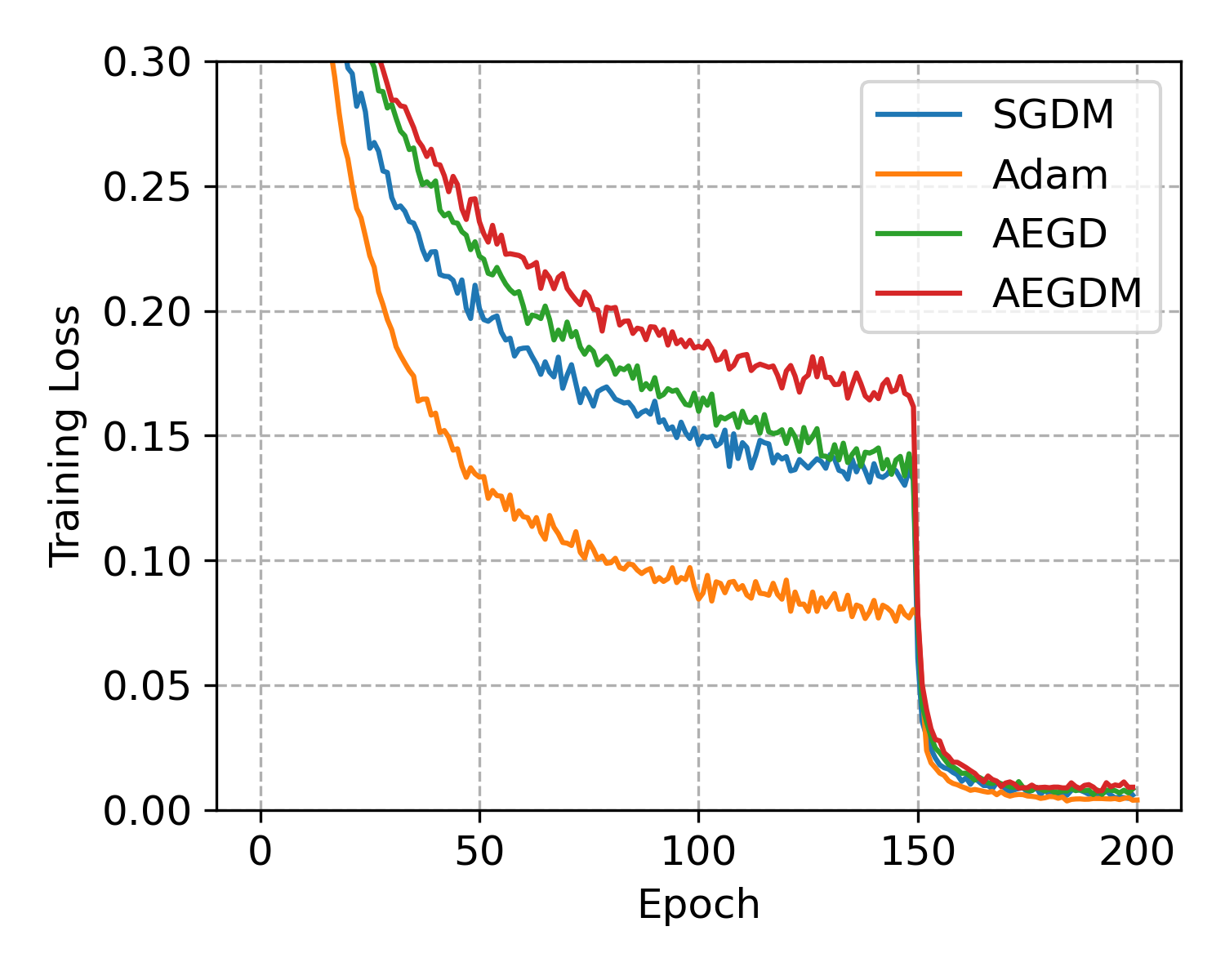}
\caption{DenseNet-121, training loss}
\end{subfigure}%
\newline
\begin{subfigure}[b]{0.33\linewidth}
\centering
\includegraphics[width=1\linewidth]{plots/cifar10_vgg16_test.png}
\caption{VGG-16, test accuracy}
\end{subfigure}%
\begin{subfigure}[b]{0.33\linewidth}
\centering
\includegraphics[width=1\linewidth]{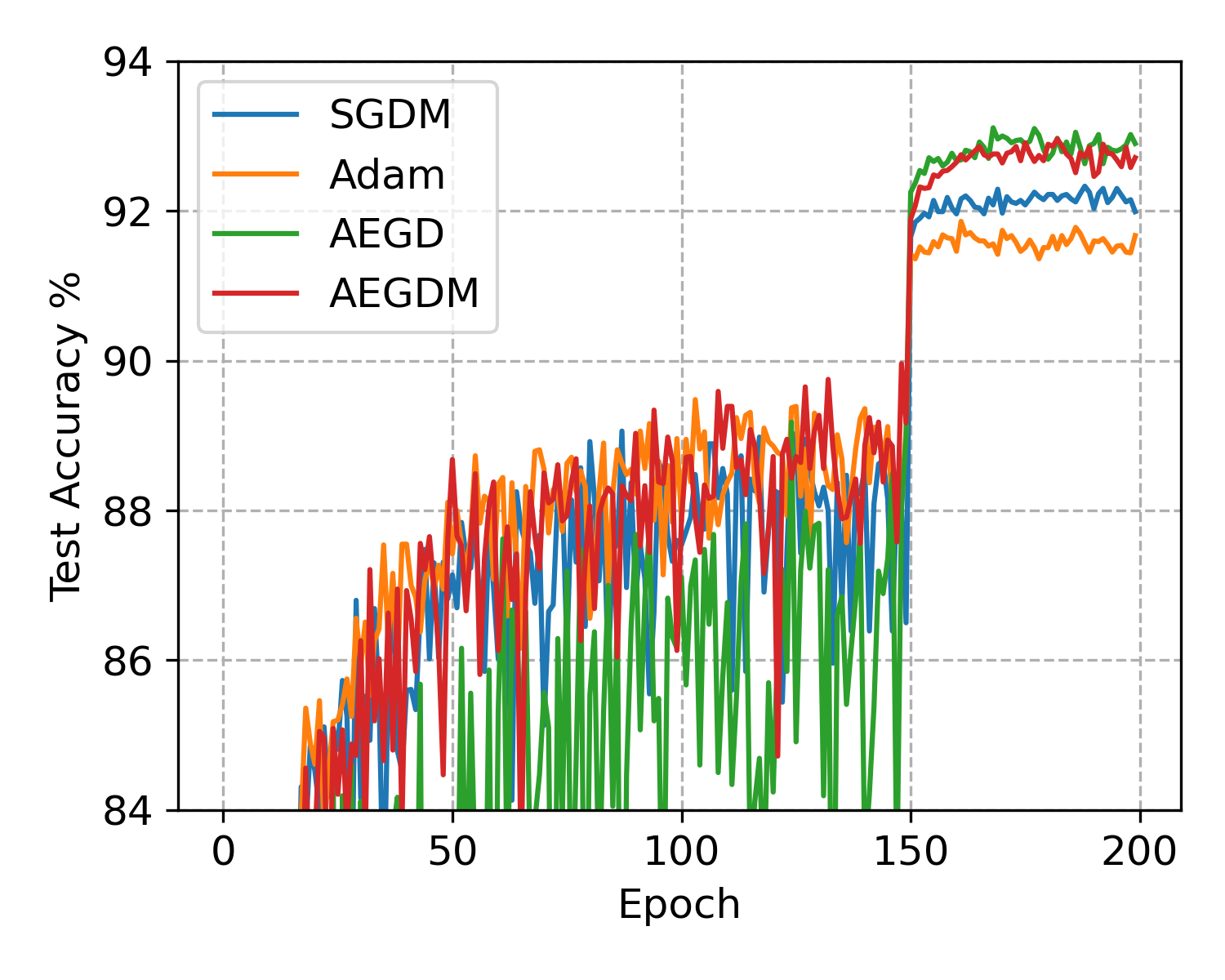}
\caption{ResNet-32, test accuracy}
\end{subfigure}%
\begin{subfigure}[b]{0.33\linewidth}
\centering
\includegraphics[width=1\linewidth]{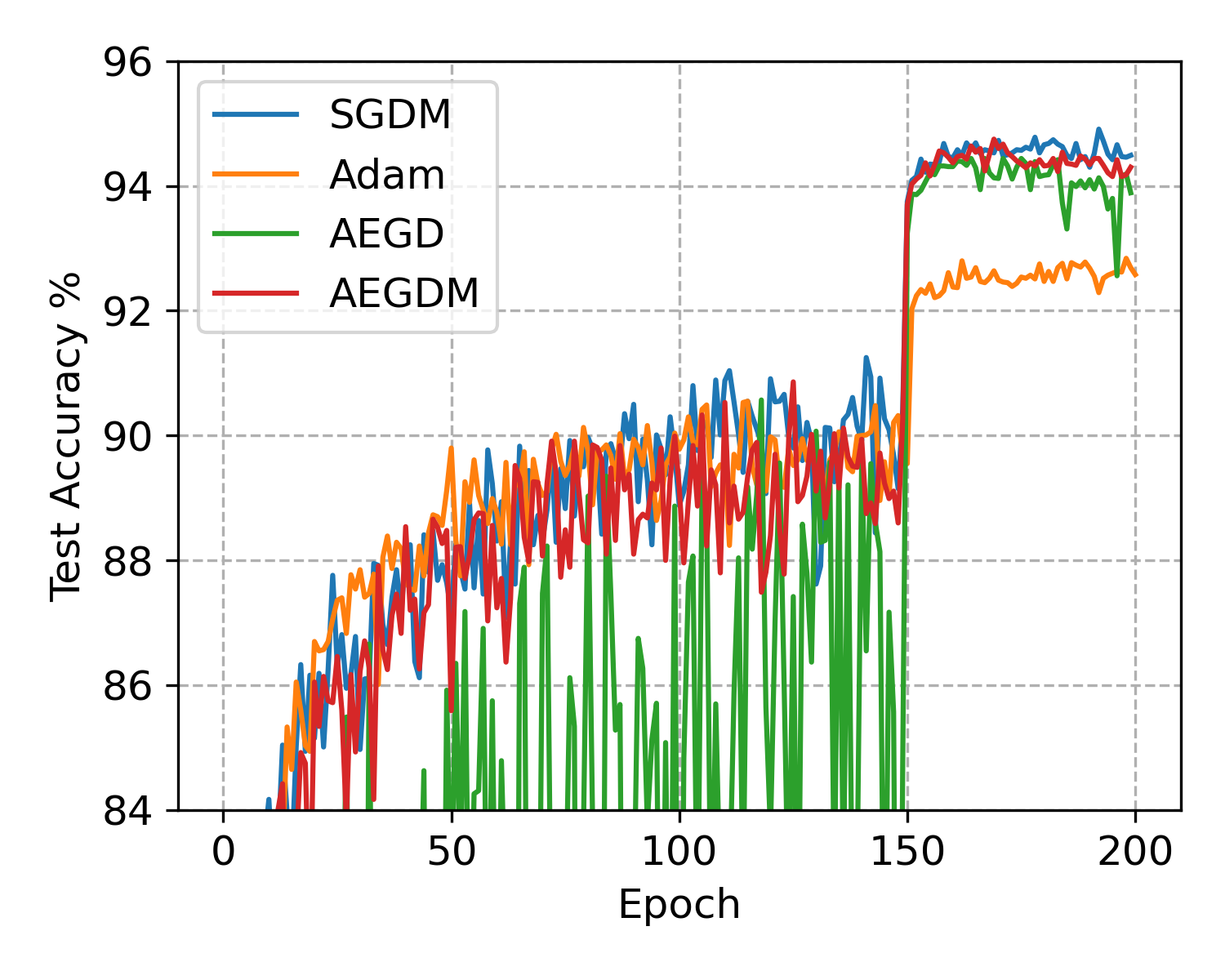}
\caption{DenseNet-121, test accuracy}
\end{subfigure}%
\captionsetup{format=hang}
\caption{Training loss and test accuracy for VGG-16, ResNet-32 and DenseNet-121 on CIFAR-10}
\label{fig:cifar10}
\end{figure}

{\bf MNIST} Figure \ref{fig:cifar100} (a) (d) show the training loss and test accuracy against epochs of each method. We see that in the training part, AEGDM and SGDM convergence faster and achieve lower training loss than AEGD and Adam.
For test accuracy, we observe an obvious fast initial progress of AEGDM, and it generalizes as well as SGDM at the end. While AEGD and Adam still have small oscillations by epoch 50. In addition,
AEGDM gives the highest test accuracy (99.3\%) among all the methods.

{\bf CIFAR-10} From Figure \ref{fig:cifar10} we can see that the oscillation of AEGD in test accuracy is significantly reduced by AEGDM. Though Adam makes rapid progress in the early stage, the generalization performance of Adam become worse than SGDM, AEGD and AEGDM after epoch 150 when the learning rate decays. In addition, we observe that AEGDM obtain better final generalization performance than SGDM in some tasks. For ResNet-32, AEGDM even surpass SGDM by $\sim 1\%$ in test accuracy. 

{\bf CIFAR-100} The results are included in Figure \ref{fig:cifar100}. We see that the overall performance of each method is similar to that on CIFAR-10. For SqueezeNet, AEGDM gives the highest test accuracy ($71.33\%$) among the four methods. For GoogleNet, AEGD outperforms SGDM by $\sim 0.5\%$ in test accuracy after learning rate decays.

\begin{figure}[ht]
\begin{subfigure}[b]{0.33\linewidth}
\centering
\includegraphics[width=1\linewidth]{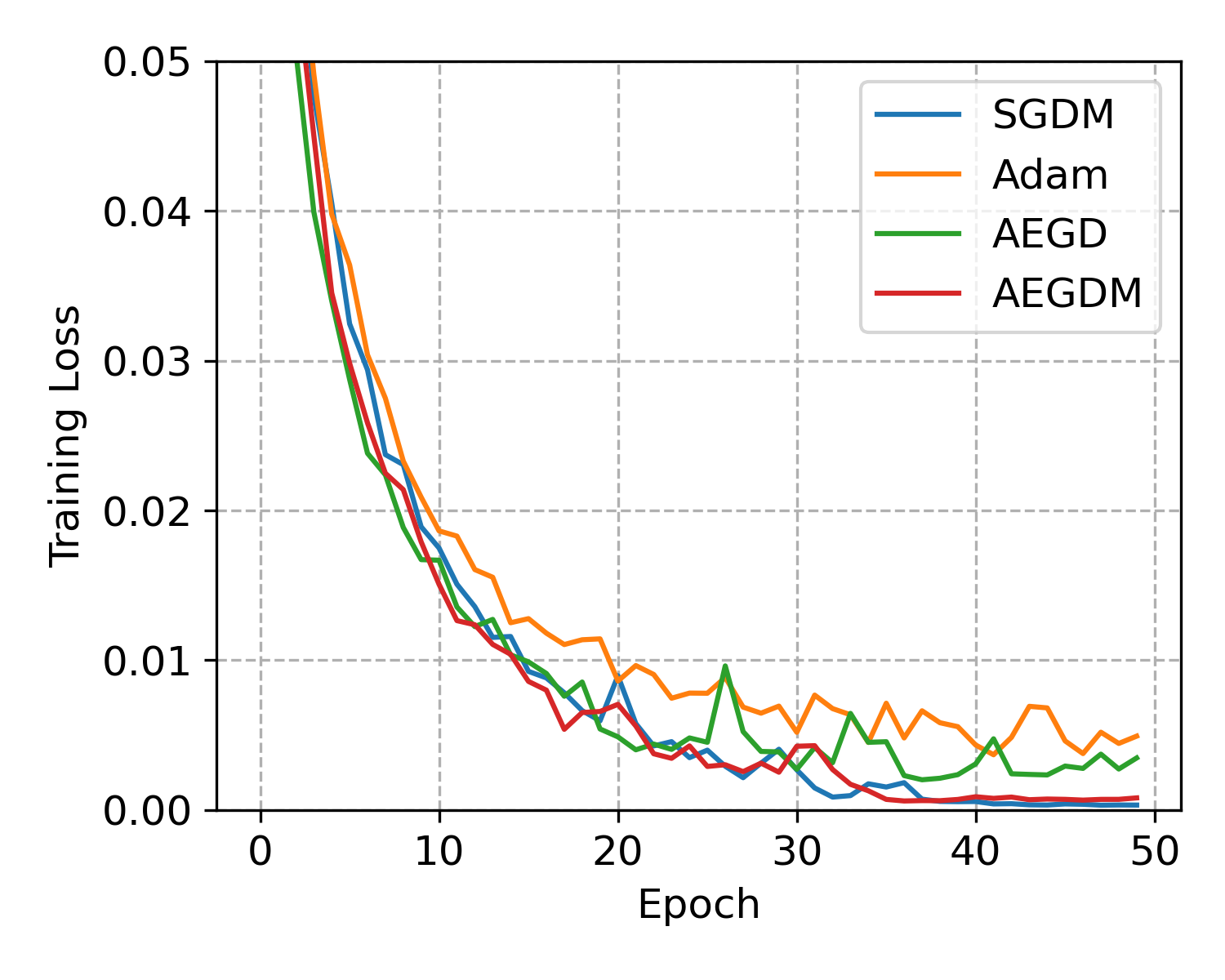}
\caption{LeNet-5, training loss}
\end{subfigure}%
\begin{subfigure}[b]{0.33\linewidth}
\centering
\includegraphics[width=1\linewidth]{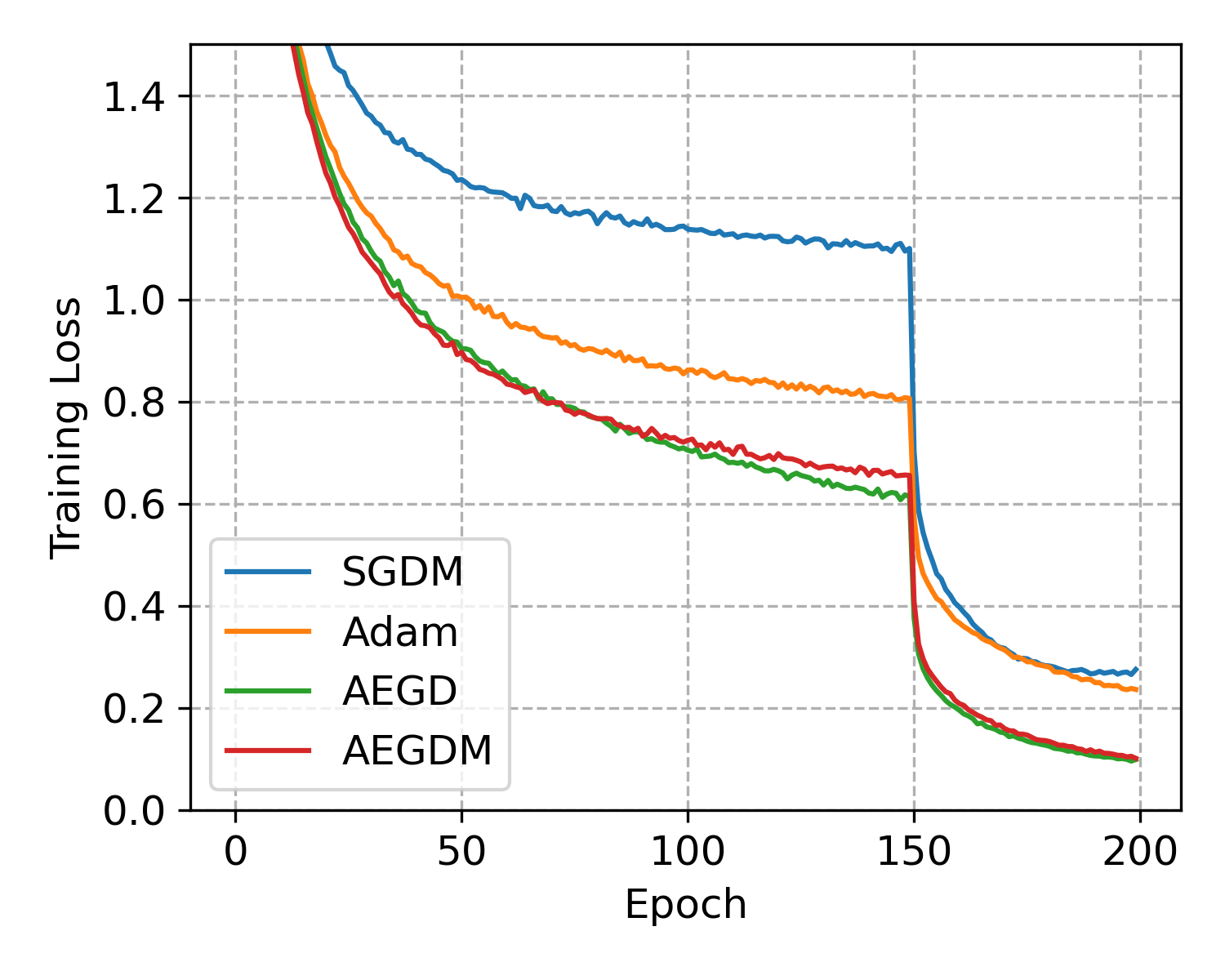}
\caption{SqueezeNet, training loss}
\end{subfigure}%
\begin{subfigure}[b]{0.33\linewidth}
\centering
\includegraphics[width=1\linewidth]{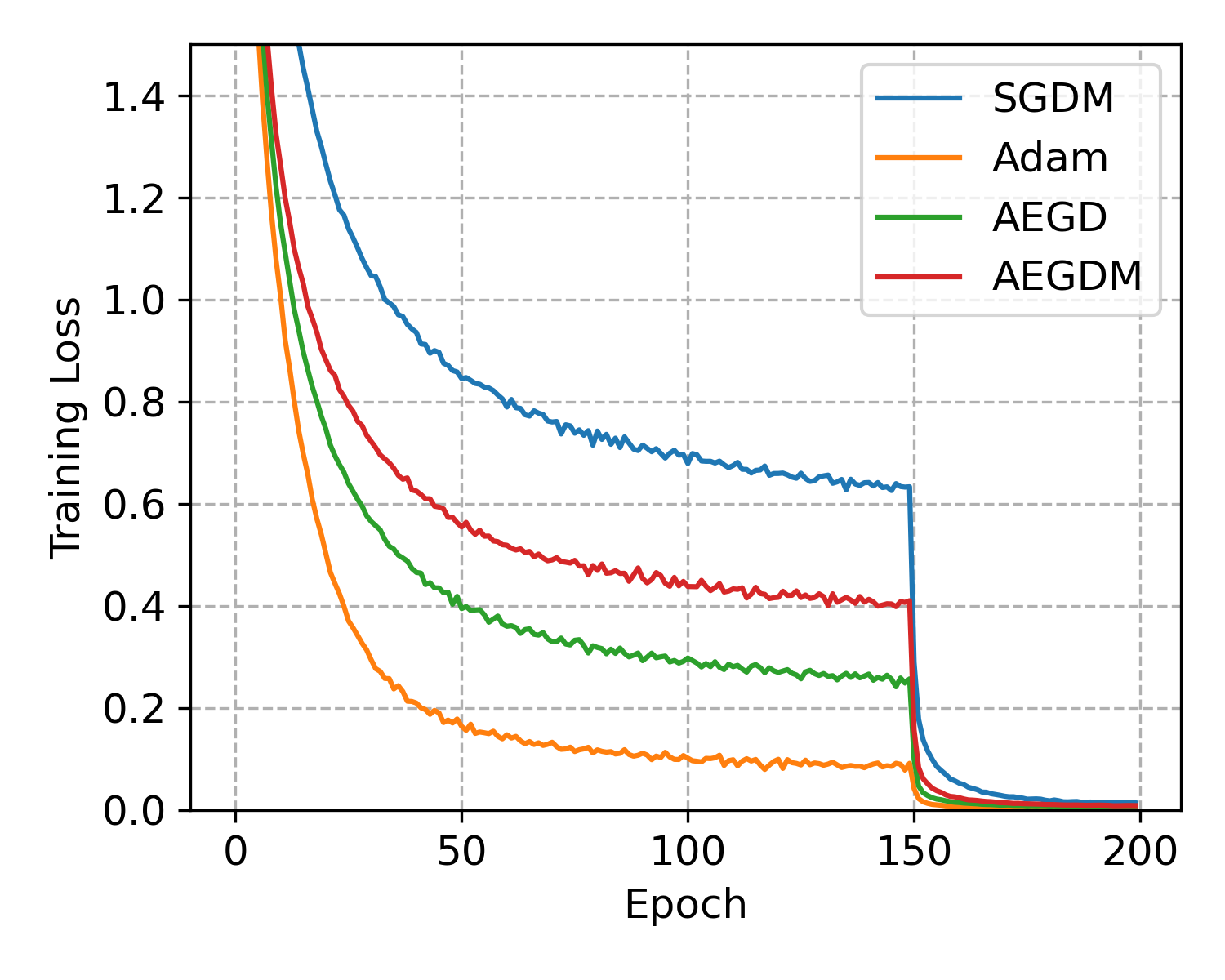}
\caption{GoogleNet, training loss}
\end{subfigure}%
\newline
\begin{subfigure}[b]{0.33\linewidth}
\centering
\includegraphics[width=1\linewidth]{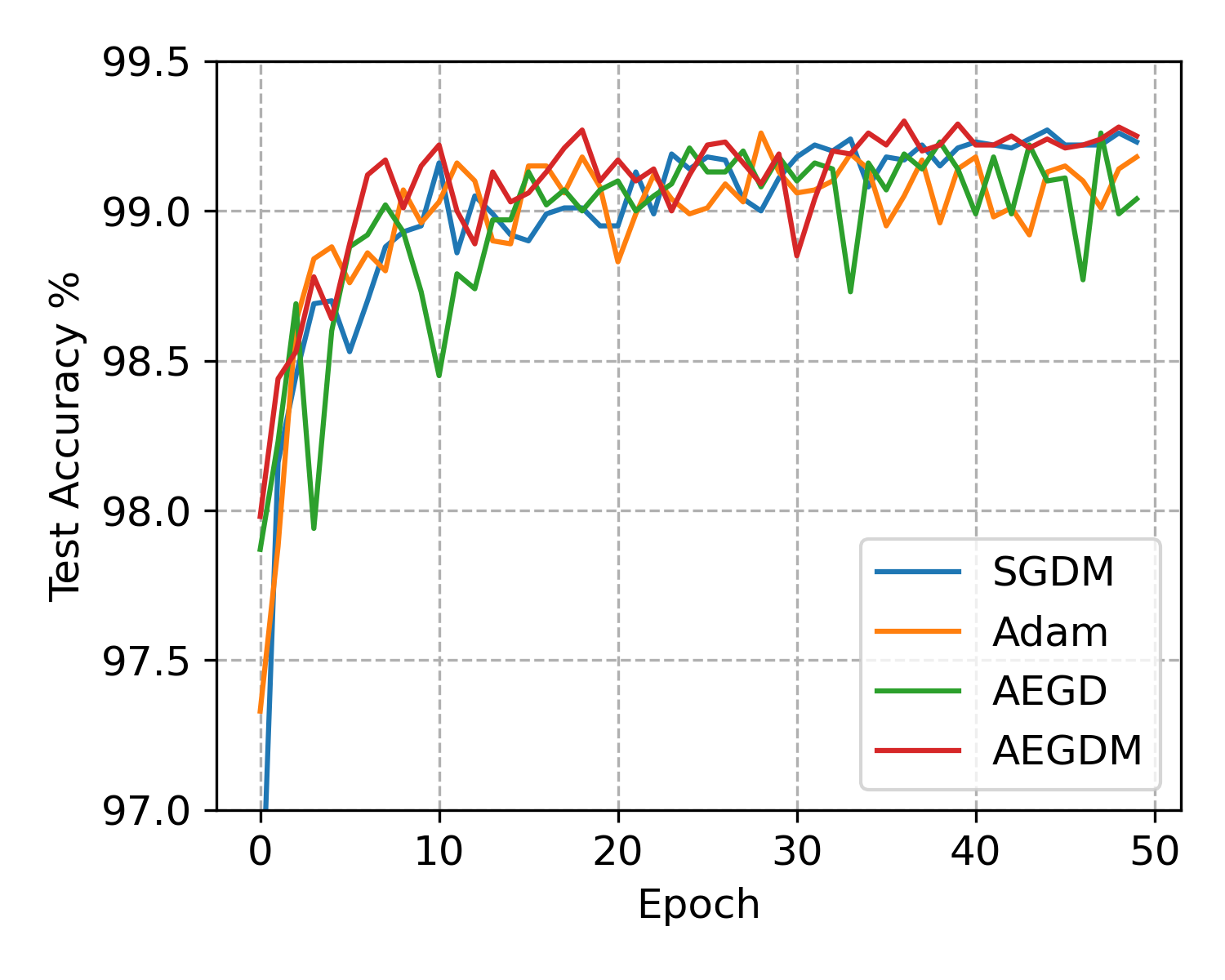}
\caption{LeNet-5, test accuracy}
\end{subfigure}%
\begin{subfigure}[b]{0.33\linewidth}
\centering
\includegraphics[width=1\linewidth]{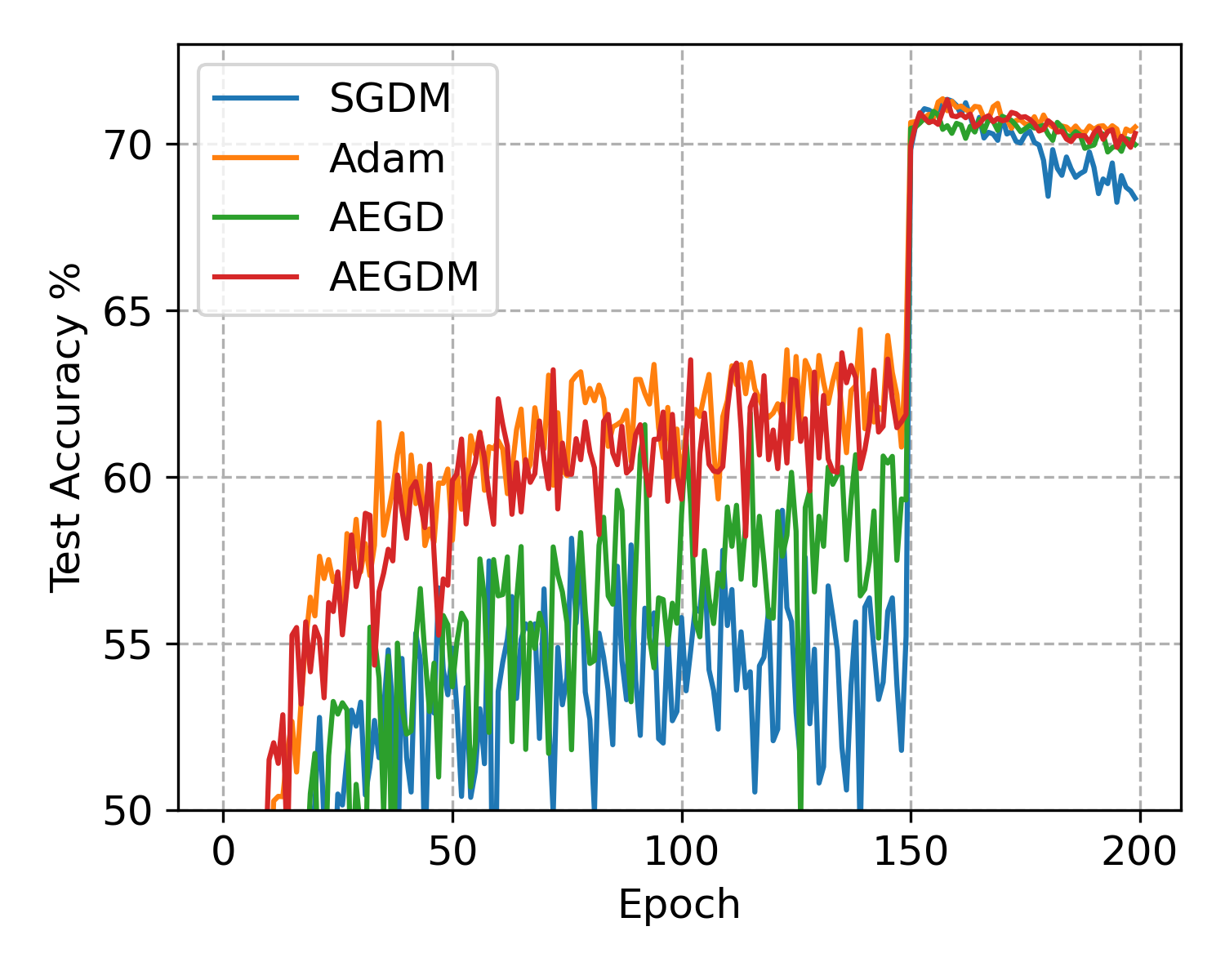}
\caption{SqueezeNet, test accuracy}
\end{subfigure}%
\begin{subfigure}[b]{0.33\linewidth}
\centering
\includegraphics[width=1\linewidth]{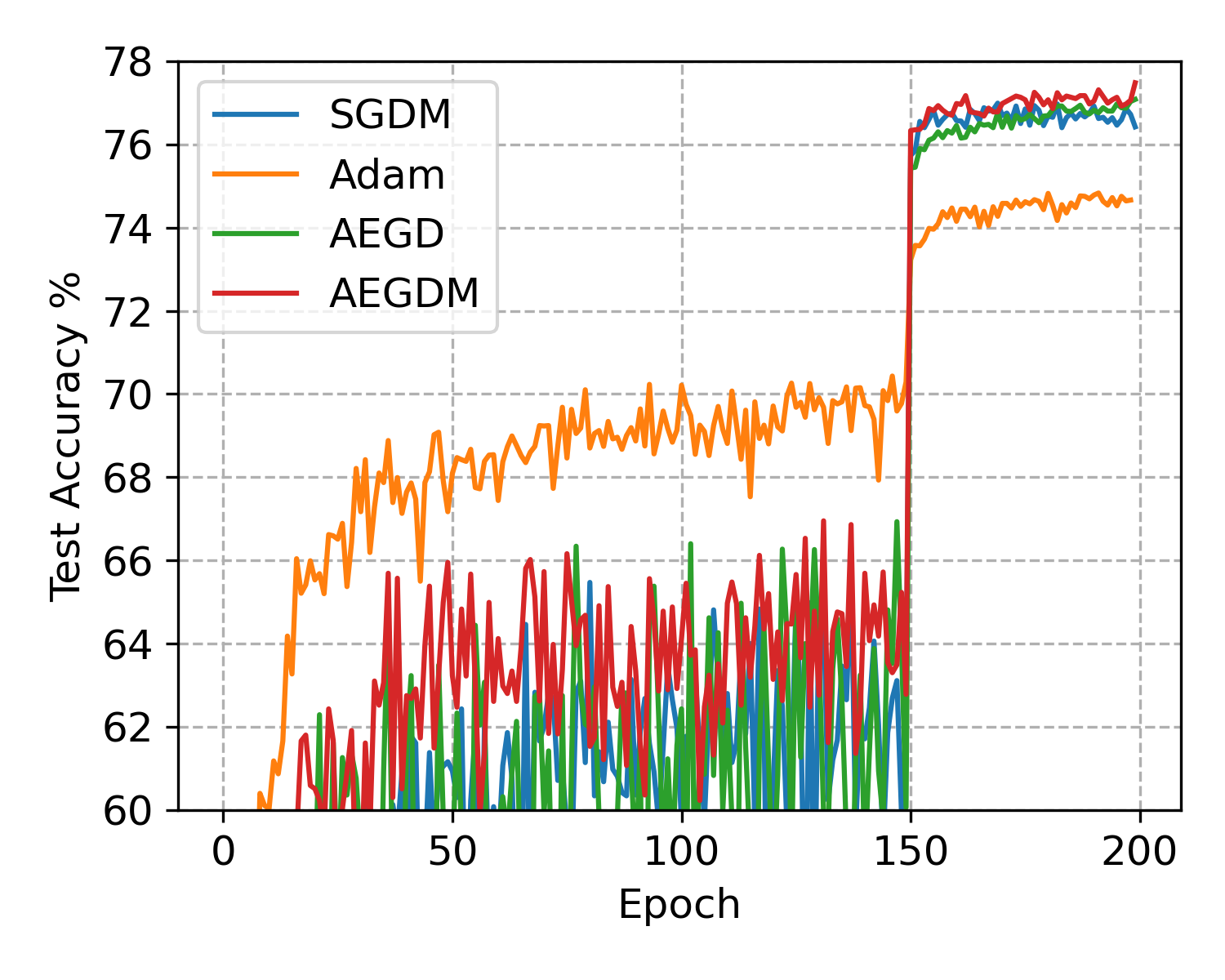}
\caption{GoogleNet, test accuracy}
\end{subfigure}%
\captionsetup{format=hang}
\caption{Training loss and test accuracy for LeNet-5 on MNIST and SqueezeNet, GoogleNet on CIFAR-100}
\label{fig:cifar100}
\end{figure}

The above results are obtained by fine tuning the base learning rate for each method in each task. While in practice, tuning hyperparameters can be tedious, thus methods require little tuning are more desirable. Therefore, we also conduct comparison (with more methods involved, including AdaBound, AdaBelief, Radam, Yogi) where the default base learning rate is used for each method in all tasks. The results and detailed setting are presented in the Appendix \ref{spep}, from which we see that AEGDM with default base learning rate (0.01) generalizes better than SGDM with default base learning rate (0.1) in all tasks.

\section{Discussion} \label{sec7} 
We have developed AEGDM,  a gradient method adapted with energy and momentum for solving stochastic optimization problems.   
The method integrates AEGD introduced in \cite{LT20} with momentum, featuring unconditional energy stability and guaranteed regret bound.  Our experiments show that
AEGDM improves AEGD by speeding up the convergence in the deterministic setting and reducing the variance in the stochastic setting. By comparison with SGDM and Adam, we show the potential of AEGD(M) in training deep neural networks to get faster convergence or better generalization performance.


The convergence rates are energy-dependent, hence it is highly desired to obtain sharper estimates on the asymptotic behavior of $r_t$ as $t\to \infty$ relative to the base learning rate $\eta$. 
The result in Theorem \ref{thm2r} as an indication of a lower bound for $r_t$ is not optimal at all. We expect a positive lower bound for $r_t$ at least for small time steps in the deterministic case. This may  be seen by analyzing the ODE system obtained from the scheme when the time step tends to zero. In fact, a global version of such system can be derived from (\ref{aegdm}) as 
\begin{subequations}\label{ode}
\begin{align}
\dot \theta &= -2rm,\\
\dot r &= -2r|v|^2,\\
0 &= -(1-\mu) m + v,
\end{align}
\end{subequations}
where $v(s)=\nabla F(\theta(s))$, with $s$ as the continuous time variable, $F(\theta)=\sqrt{f(\theta)+c}$. Observe that 
\begin{align*}
\frac{d}{ds}\left[ r(s)-(1-\mu)F(\theta(s))\right] 
&=-2r|v|^2-(1-\mu)\langle \nabla F(\theta),\dot\theta\rangle\\
&=-2r|v|^2-(1-\mu)\langle v, -2rm\rangle=0. 
\end{align*}
Using the choice $r_0=F(\theta_0)$, we obtain 
$$
r(s)=(1-\mu) F(\theta(s))+\mu F(\theta_0).
$$
Hence for any $s>0$, 
$$
r(s)\geq (1-\mu) \min F(\theta)+\mu F(\theta_0)\geq \min F(\theta)>0. 
$$
We leave more refined analysis in future work. 
\bigskip

{\noindent\bf Data availability:} 
The data that support the findings of this study are
publicly available online at \url{http://yann.lecun.com/exdb/mnist/} and \url{https://www.cs.toronto.edu/~kriz/cifar.html}.

\section*{Acknowledgments}  
This work was supported by the National Science Foundation under Grant DMS1812666.

\bibliographystyle{amsplain}
\bibliography{ref}




\appendix
In this appendix, we present technical proofs of theoretical results in this work.
\section{Auxiliary lemmas and notation} 
\begin{lemma} \label{lemA1}
Set $G_i(T, \mu)=
\sum_{t=1}^{T}r_{t, i}m_{t, i}^2$.
Then 
\begin{equation}\label{GTM}
G_i(T, \mu)\leq 
\frac{1}{(1-\mu)^2}G_i(T, 0)
\end{equation}
provided $\mu<1$.
\end{lemma}
\begin{proof} Note that using $m_{0, i}=0$ we may rewrite $G_i$ as 
$$
G_i(T, \mu)=
\sum_{t=0}^{T}r_{t, i}m_{t, i}^2
$$
with $G_i(0, \mu)=0$. For any $\epsilon>0$ with $(1+\epsilon)\mu^2<1$, 
\begin{align*}
G_i(T, \mu) &= \sum_{t=0}^{T-1}r_{t+1, i} m^2_{t+1, i}=
\sum_{t=0}^{T-1}r_{t+1, i}(\mu m_{t, i}+ v_{t, i})^2\\
& \leq \sum_{t=0}^{T-1}((1+\epsilon) \mu^2 r_{t, i}m_{t, i}^2 + (1+\epsilon^{-1})r_{t+1, i}v_{t, i}^2)\qquad (\text{since} \quad r_{t+1,i}\leq r_{t,i}) \\
& \leq (1+\epsilon)\mu^2 G_i(T-1, \mu)+ (1+\epsilon^{-1})G_i(T, 0) \\
& \leq ((1+\epsilon)\mu^2)^{T} G_i(0, \mu) +(1+\epsilon^{-1})G_i(T, 0)(1+(1+\epsilon)\mu^2 +\cdots +((1+\epsilon)\mu^2)^{T-1})\\
& \leq 
(1+\epsilon^{-1})G_i(T, 0) \frac{1}{1-(1+\epsilon)\mu^2}.  
\end{align*}
Since $0\leq \mu<1$, we may take $\epsilon=\frac{1-\mu}{\mu}>0$ so that $(1+\epsilon)\mu^2=\mu<1$, and $1+\epsilon^{-1}=\frac{1}{1-\mu}$, hence (\ref{GTM}).
\end{proof}

\begin{lemma}\label{lemA2}
For $0\leq \mu<1$, 
\begin{equation}\label{gtm}
G_i(T, \mu) \leq \frac{r_{0, i}}{2\eta(1-\mu)^2}.
\end{equation}
\end{lemma}
\begin{proof}
From (\ref{aegdm}c), it follows 
\begin{equation*}
r_{t,i} - r_{t+1,i} 
= 2\eta r_{t+1,i}v_{t,i}^2.
\end{equation*}
Taking summation over $t$ from $0$ to $T-1$ and using telescopic cancellation, we have 
$$
r_{0,i} - r_{T,i} = 2\eta\sum_{t=0}^{T-1} r_{t+1,i}v^2_{t,i}.
$$
Rearrange the above to get
\begin{equation}\label{gt0}
G_i(T, 0)=\sum_{t=0}^{T-1} r_{t+1, i}v_{t, i}^2
\leq \frac{r_{0,i}}{2\eta}.
\end{equation}
By Lemma \ref{lemA1}, we get (\ref{gtm}). 
\end{proof}
For the proofs of Theorem \ref{prop1} and \ref{thm1}, we introduce notation, 
\begin{equation}\label{Ft}
F_t:=\sqrt{f_t(\theta_t)+c}.    
\end{equation}
The initial data for $r_i$ is taken as $r_{0, i}=F_0.$ With such choice 
\begin{equation}\label{GTmu}
 G(T, \mu)=\sum_{i=1}^nG_i(T, \mu)\leq \frac{nF_0}{2\eta(1-\mu)^2}.   
\end{equation}

For the proofs of Theorem \ref{prop2}, Theorem \ref{thm2c} and Theorem \ref{thm2r}, we introduce notation 
\begin{equation}\label{Ft2}
\tilde F_t:=\sqrt{f(\theta_t; \xi_t)+c}.    
\end{equation}
The initial data for $r_i$ is taken as $r_{0, i}=\tilde F_0.$ With such choice,  
\begin{equation}\label{GTmu2}
G(T,\mu)=\sum_{i=1}^{n}G_i(T,\mu)=\frac{n\tilde F_0}{2\eta(1-\mu)^2}.
\end{equation}

\begin{lemma}\label{lem3}
Under the assumptions  in Theorem \ref{thm2c}, we have for all $t\in [T]$, 
\begin{enumerate}[label=(\roman*)]
\item $\|\nabla f(\theta_t)\|_\infty\leq G_\infty$.
\item $\E[(\tilde F_t)^2]= F^2(\theta_t)=f(\theta_t)+c$.
\item 
$\E[\tilde F_t]\leq F(\theta_t)$. In particular, $\E[r_{0,i}]= \E[\tilde F_0]\leq F(\theta_0)$.
\item $\E[|F(\theta_t)-\tilde F_t|]\leq \frac{1}{2a^{1/2}}\sigma_f$.
\item $\E[\|\nabla F(\theta_t)-v_t\|^2_2]\leq \frac{G^2_\infty}{8a^3}\sigma^2_f+\frac{1}{2a}\sigma^2_g.$
\end{enumerate}
\end{lemma}
\begin{proof}
(i) By assumption $\|g_t\|_\infty\leq G_\infty$, we have
$$\|\nabla f(\theta_t)\|_\infty=\|\E[g_t]\|_\infty\leq\E[\|g_t\|_\infty]\leq G_\infty.$$
(ii) This follows from the unbiased sampling of 
$$
f(\theta_t)=\E_{\xi_t}[ f(\theta_t; \xi_t)].
$$
(iii) By Jensen's inequality, we have
$$\E[\tilde F_t] 
\leq\sqrt{\E[\tilde F_t^2]}=\sqrt{F(\theta_t)^2}=F(\theta_t).$$
(iv) By the assumption $0<a\leq   f(\theta_t;\xi_t)+c=\tilde F_t^2$, we have
\begin{align*}
&\quad \E[|F(\theta_t)-\tilde F_t|]
\leq \E\Bigg[\bigg|\frac{f(\theta_t)-f(\theta_t;\xi_t)}{F(\theta_t)+\tilde F_t}\bigg|\Bigg] \leq \frac{1}{2a^{1/2}}\E[|f(\theta_t)-f(\theta_t;\xi_t)|] 
\leq \frac{1}{2a^{1/2}}\sigma_f.
\end{align*}
(v) By the definition of $F(\theta)$ and $v_t$ in (\ref{aegdm}a), we have
\begin{align*}
\|\nabla F(\theta_t)-v_t\|^2_2
&= \bigg\|\frac{\nabla f(\theta_t)}{2F(\theta_t)}-\frac{g_t}{2\tilde F_t}\bigg\|^2_2\\
&= \frac{1}{4}\bigg\|\frac{\nabla f(\theta_t)(\tilde F_t-F(\theta_t)) }{F(\theta_t)\tilde F_t}+\frac{\nabla f(\theta_t)-g_t}{\tilde F_t}\bigg\|^2_2\\
&\leq \frac{1}{2} \bigg\|\frac{\nabla f(\theta_t)(\tilde F_t-F(\theta_t)) }{F(\theta_t)\tilde F_t}\bigg\|^2_2 + \frac{1}{2}\bigg\|\frac{\nabla f(\theta_t)-g_t}{\tilde F_t}\bigg\|^2_2\\
&\leq \frac{G^2_\infty}{2a^{2}}|\tilde F_t-F(\theta_t)|^2+\frac{1}{2a}\|\nabla f(\theta_t)-g_t\|^2_2, 
\end{align*}
where 
both the gradient bound and the assumption that $0<a\leq f(\theta_t;\xi_t)+c=\tilde F^2_t$ are essentially used.  Take an expectation to get 
\begin{align*}
\E[\|\nabla F(\theta_t)-v_t\|^2_2]\leq \frac{G^2_\infty}{2a^{2}}\E[|\tilde F_t-F(\theta_t)|^2]+\frac{1}{2a}\E[\|\nabla f(\theta_t)-g_t\|^2_2].
\end{align*}
Similar to the proof for ($iv$), we have
$$
\E[|\tilde F_t-F(\theta_t)|^2]\leq \frac{1}{4a}\sigma^2_f.
$$
This together with the variance assumption  for $g_t$ gives
\begin{equation*}
\E[\|\nabla F(\theta_t)-v_t\|^2_2]\leq \frac{G^2_\infty}{8a^3}\sigma^2_f+\frac{1}{2a}\sigma^2_g.
\end{equation*}
\end{proof}

\section{Proof of Theorem \ref{prop1}}\label{pfp1}
The decreasing of $r_{t,i}$ can be easily seen from (\ref{aegdm}c) since $r_{0,i}>0$ and $1+2\eta v^2_{t,i}\geq1$ for all $t\geq1$ and $i\in[n]$.
Using (\ref{gtm}), we have
\begin{align}\notag
\quad\sum_{t=0}^{T-1}\|\theta_{t+1}-\theta_t\|^2_2
& =\sum_{i=1}^n\sum_{t=0}^{T-1} (-2\eta r_{t+1, i}m_{t+1, i})^2\\\notag
&\leq 4\eta^2 \sum_{i=1}^n\sum_{t=0}^{T-1} r_{0, i} r_{t+1, i}m^2_{t+1, i}\\\notag
& = 4\eta^2 G(T, \mu)F_0 
\leq \frac{2 \eta n}{(1-\mu)^2}(f_0(\theta_0)+c).
\end{align}
Using the Cauchy-Schwarz inequality,  we get
\begin{align*}
&\quad\sum_{t=0}^{T-1}|v_{t,i}|=\sum_{t=0}^{T-1}\frac{1}{\sqrt{r_{t+1,i}}}\sqrt{r_{t+1,i}}|v_{t,i}|\\
&\leq\Bigg(\sum_{t=0}^{T-1}\frac{1}{r_{t+1,i}}\Bigg)^{1/2}\Bigg(\sum_{t=0}^{T-1}r_{t+1,i}v_{t, i}^2\Bigg)^{1/2}\\
& \leq \Bigg(G_i(T, 0)\Bigg)^{1/2}\Bigg(\frac{T}{r_{T,i}}\Bigg)^{1/2}.    
\end{align*}
The desired estimate \eqref{T} follows by using (\ref{gt0}).

\section{Proof of Theorem \ref{thm1}}\label{pf1}
By convexity of $f_t$, the regret can be bounded by
\begin{align}\label{rt}
R(T) &= \sum_{t=0}^{T-1}f_t(\theta_t)-f_t(\theta^*)
\leq \sum_{t=0}^{T-1} \nabla f_t(\theta_t)^\top (\theta_t-\theta^*)=\sum_{t=0}^{T-1}\sum_{i=1}^{n}
\partial_i f_{t}(\theta_t) (\theta_{t,i}-\theta^*_i).
\end{align}
Using the update rule \eqref{aegdm}, we have for $i\in[n]$,
$$
\theta_{t+1,i} = \theta_{t,i} - 2\eta r_{t+1,i}m_{t+1,i}=\theta_{t,i} - 2\eta r_{t+1,i}(\mu m_{t,i}+v_{t,i}), 
$$
which upon subtraction of $\theta^*_i$ and squaring both sides yields 
$$
(\theta_{t+1,i}-\theta^*_i)^2 = (\theta_{t,i}-\theta^*_i)^2 - 4\eta r_{t+1,i}(\mu m_{t,i}+v_{t,i})(\theta_{t,i}-\theta^*_i) + 4\eta^2 r^2_{t+1,i}m^2_{t+1,i}.
$$
Rearranging we get
\begin{align*}
4\eta r_{t+1,i}v_{t, i}(\theta_{t,i}-\theta^*_i) = &\Big((\theta_{t,i}-\theta^*_i)^2 - (\theta_{t+1,i}-\theta^*_i)^2\Big) - 4\eta r_{t+1,i}\mu m_{t,i}(\theta_{t,i}-\theta^*_i) \\
&+ 4\eta^2 r^2_{t+1,i}m^2_{t+1,i}.
\end{align*}
Note that $4\eta r_{t+1,i}v_{t, i}=2\eta r_{t+1,i}\partial_i f_{t}/ F_t$, hence the above can be rewritten as  
\begin{align*}
&\quad\partial_i f_{t}(\theta_t)(\theta_{t,i}-\theta^*_i) \\
&= \frac{F_t}{2\eta r_{t+1,i}}\Big((\theta_{t,i}-\theta^*_i)^2 - (\theta_{t+1,i}-\theta^*_i)^2\Big) + 2\mu F_t  m_{t,i}(\theta^*_i-\theta_{t,i}) + 2 \eta F_t r_{t+1,i}m^2_{t+1,i}\\
&= \frac{F_t}{2\eta r_{t+1,i}}(\theta_{t,i}-\theta_{t+1,i})\Big((\theta_{t,i}-\theta^*_i)+(\theta_{t+1,i}-\theta^*_i)\Big) + 2\mu F_t  m_{t,i}(\theta^*_i-\theta_{t,i}) + 2 \eta F_t r_{t+1,i}m^2_{t+1,i}\\
&= F_t m_{t+1,i}\Big((\theta_{t,i}-\theta^*_i)+(\theta_{t+1,i}-\theta^*_i)\Big) + 2\mu F_t  m_{t,i}(\theta^*_i-\theta_{t,i}) + 2 \eta F_t r_{t+1,i}m^2_{t+1,i}.
\end{align*}
Using (\ref{rt}), we have
\begin{equation}\label{regret1} 
\begin{aligned}
R(T) &\leq \sum_{t=0}^{T-1}\sum_{i=1}^{n}F_t m_{t+1,i}\Big((\theta_{t,i}-\theta^*_i)+(\theta_{t+1,i}-\theta^*_i)\Big) \\
&\quad +\sum_{t=0}^{T-1}\sum_{i=1}^{n} 2\mu F_t  m_{t,i}(\theta^*_i-\theta_{t,i})
+\sum_{t=0}^{T-1}\sum_{i=1}^{n}2\eta F_tr_{t+1,i}m^2_{t+1,i}.
\end{aligned}  
\end{equation}
Introduce $M(T):=\sum_{t=0}^{T-1}\sum_{i=1}^{n}  |m_{t+1,i}|$ and using the bound $F_t\leq B^{1/2}$, we can estimate $R(T)$ as follows: 
\begin{align*}
  R(T) & \leq 2B^{1/2}D_\infty (M(T)+\mu M(T-1))+2\eta B^{1/2}G(T, \mu) \\
  & \leq 2(1+\mu) B^{1/2}D_\infty M(T)+2\eta B^{1/2}G(T, \mu). 
\end{align*}
Note that $G(T, \mu)$ is bounded by \eqref{GTmu} and $M(T)$ can be bounded by
\begin{align*}
M(T)
& = \sum_{t=0}^{T-1}\sum_{i=1}^{n}  |m_{t+1,i}|\\
&\leq \left(\sum_{t=0}^{T-1}\sum_{i=1}^{n} 
r_{t+1, i}|m_{t+1, i}|^2 \right)^{1/2}\left(\sum_{t=0}^{T-1}\sum_{i=1}^{n} \frac{1}{r_{t+1, i}}\right)^{1/2} \\
& \leq  \sqrt{G(T, \mu)} \left(\sum_{i=1}^{n} \frac{T}{r_{T, i}}\right)^{1/2}.
\end{align*}

Connecting all the above estimates, we obtain
$$
R(T) \leq  C_1 \Bigg(\sum_{i=1}^n \frac{1}{r_{T, i}}\Bigg)^{1/2}\sqrt{T}+ C_2,
$$
where 
$$
C_1= 2 (1+\mu) B^{1/2} D_\infty\sqrt{G(T, \mu)}, \quad C_2=2\eta B^{1/2} G(T, \mu).  
$$
The regret bound \eqref{regret1} follows by using \eqref{GTmu}.
This completes the proof of the regret bound.

\section{Proof of Theorem \ref{prop2}}\label{pf2c-}
The proof is entirely similar to the proof of Theorem \ref{prop1}, with the use of expectation.

\section{Proof of Theorem \ref{thm2c}}\label{pf2c}

Since $f$ is $L$-smooth, we have
\begin{align}\label{fL}
f(\theta_{t+1})\leq f(\theta_t)+\nabla f(\theta_t)^\top(\theta_{t+1}-\theta_t)+\frac{L}{2}\|\theta_{t+1}-\theta_t\|^2_2.
\end{align}
Denoting $\eta_t=\eta/\tilde F_t$, we rewrite the second term in the RHS of (\ref{fL}) as
\begin{align}\notag
&\quad\nabla f(\theta_t)^\top(\theta_{t+1}-\theta_t)\\\notag
&= \nabla f(\theta_t)^\top(-2\eta r_{t+1} m_{t+1})\\\notag
&= -2\eta\nabla f(\theta_t)^\top  r_{t+1}(\mu m_{t} + v_{t})\\\notag
&= -2\eta\nabla f(\theta_t)^\top r_{t+1} v_{t} -2\mu\eta\nabla f(\theta_t)^\top r_{t+1}m_{t}\\\notag
&= -\nabla f(\theta_t)^\top \eta_tr_{t+1} g_{t} -2\mu\eta\nabla f(\theta_t)^\top r_{t+1}m_{t}\qquad(\text{since}\quad g_t=2\tilde F_t v_t)\\\label{T2}
&= -\nabla f(\theta_t)^\top \eta_{t-1}r_{t} g_{t} + \nabla f(\theta_t)^\top (\eta_{t-1}r_{t}-\eta_tr_{t+1}) g_{t} -2\mu\eta\nabla f(\theta_t)^\top r_{t+1}m_{t}.
\end{align}
We further bound the second term and third term in the RHS of (\ref{T2}) separately. For the second term, we have
\begin{align}\notag
&\quad\nabla f(\theta_t)^\top (\eta_{t-1}r_{t}-\eta_tr_{t+1}) g_{t}\\\notag
&=\nabla f(\theta_t)^\top \eta_{t-1}(r_{t}-r_{t+1}) g_{t}+ \nabla f(\theta_t)^\top (\eta_{t-1}-\eta_t)r_{t+1}g_{t}\\\notag
&=\nabla f(\theta_t)^\top \eta_{t-1}(r_{t}-r_{t+1}) g_{t}+(\eta_{t-1}-\eta_t)g_t^\top r_{t+1}g_{t}\\\notag
&\quad+(\eta_{t-1}-\eta_t)(\nabla f(\theta_t)-g_t)^\top r_{t+1}g_{t}\\\notag
&\leq \eta_{t-1}\|\nabla f(\theta_t)\|_\infty \|r_{t}-r_{t+1}\|_{1,1} \|g_{t}\|_\infty + |\eta_{t-1}-\eta_t|g_t^\top r_{t+1}g_{t}\\\notag
&\quad+|\eta_{t-1}-\eta_t|\cdot |(\nabla f(\theta_t)-g_t)^\top r_{t+1}g_{t}|\\\notag
&\leq (\eta G^2_\infty/\sqrt{a})(\|r_t\|_{1,1}-\|r_{t+1}\|_{1,1})+(2\eta/\sqrt{a})g_t^\top r_{t+1}g_{t}\\\label{T21}
&\quad+(2\eta/\sqrt{a})|(\nabla f(\theta_t)-g_t)^\top r_{t+1}g_{t}|.
\end{align}
The fourth inequality holds because for a positive diagonal matrix $A$, $x^\top Ay\leq\|x\|_\infty\|A\|_{1,1}\|y\|_\infty$, where $\|A\|_{1,1}=\sum_{i}a_{ii}$. The last inequality follows from  $r_{t+1,i}\leq r_{t,i}$ for $i\in[n]$ and (i) in Lemma \ref{lem3}.

For the third term in the RHS of (\ref{T2}), we have
\begin{align}\notag
-2\mu\eta\nabla f(\theta_t)^\top r_{t+1}m_{t}
&= -2\mu\eta g_{t}^\top r_{t+1}m_{t} + 2\mu\eta(g_t-\nabla f(\theta_t))^\top r_{t+1}m_{t}\\\label{T22}
&\leq \mu\eta g_t^\top r_{t+1}g_t + \mu\eta m_{t}^\top r_{t+1}m_{t} + 2\mu\eta|(g_t-\nabla f(\theta_t))^\top r_{t+1}m_{t}|,
\end{align}
where the inequality follows from that for a positive diagonal matrix $A$, $x^\top Ay\leq \frac{1}{2}x^\top Ax+\frac{1}{2}y^\top Ay$. 
Connecting (\ref{T21}) and (\ref{T22}), we can further bound (\ref{T2}) by
\begin{equation}\label{T2s}
\begin{aligned}
\nabla f(\theta_t)^\top(\theta_{t+1}-\theta_t)
&\leq -\nabla f(\theta_t)^\top \eta_{t-1}r_{t} g_{t} + (\eta G^{2}_\infty/\sqrt{a})(\|r_{t}\|_{1,1}-\|r_{t+1}\|_{1,1}) \\
&\quad+ (2\eta/\sqrt{a}+\mu\eta) g_{t}^\top r_{t+1}g_{t} + \mu\eta m_{t}^\top r_{t+1}m_{t}\\ 
&\quad+ (2\eta/\sqrt{a})|(\nabla f(\theta_t)-g_t)^\top r_{t+1}g_{t}|+2\mu\eta|(g_t-\nabla f(\theta_t))^\top r_{t+1}m_{t}|.
\end{aligned}
\end{equation}
Substituting (\ref{T2s}) into \eqref{fL} and rearranging to get
\begin{align*}
\nabla f(\theta_t)^\top \eta_{t-1}r_{t} g_{t}
&\leq (f(\theta_t)-f(\theta_{t+1})) + (\eta G^{2}_\infty/\sqrt{a})(\|r_{t}\|_{1,1}-\|r_{t+1}\|_{1,1}) \\
&\quad+ (2\eta/\sqrt{a}+\mu\eta) g_{t}^\top r_{t+1}g_{t} + \mu\eta m_{t}^\top r_{t+1}m_{t}\\ 
&\quad+ (2\eta/\sqrt{a})|(\nabla f(\theta_t)-g_t)^\top r_{t+1}g_{t}|+2\mu\eta|(g_t-\nabla f(\theta_t))^\top r_{t+1}m_{t}|\\
&\quad+ \frac{L}{2}\|\theta_{t+1}-\theta_t\|^2_2.
\end{align*}
Taking an conditional expectation on $(\theta_t,r_t)$, we have
\begin{equation}\label{et}
\begin{aligned}
&\quad \nabla f(\theta_t)^\top \eta_{t-1}r_{t}\nabla f(\theta_t)
=\E_{\xi_t}\bigg[\nabla f(\theta_t)^\top \eta_{t-1}r_{t}g_t\bigg]\\
&\leq \E_{\xi_t}\Bigg[(f(\theta_t)-f(\theta_{t+1})) + (\eta G^{2}_\infty/\sqrt{a})(\|r_{t}\|_{1,1}-\|r_{t+1}\|_{1,1})\\
&\quad\quad+(2\eta/\sqrt{a}+\mu\eta) g_{t}^\top r_{t+1}g_{t} + \mu\eta m_{t}^\top r_{t+1}m_{t}\\ 
&\quad\quad+(2\eta/\sqrt{a})|(\nabla f(\theta_t)-g_t)^\top r_{t+1}g_{t}|+2\mu\eta|(g_t-\nabla f(\theta_t))^\top r_{t+1}m_{t}|\\
&\quad\quad+\frac{L}{2}\|\theta_{t+1}-\theta_t\|^2_2\Bigg],
\end{aligned}
\end{equation}
where the assumption $\E_{\xi_t}[g_t]=\nabla f(\theta_t)$ is used in the first equality. Since $\xi_1,...,\xi_t$ are independent random variables, we set $\E=\E_{\xi_1}\E_{\xi_2}...\E_{\xi_T}$ and take a summation on (\ref{et}) over $t$ from 0 to $T-1$ to get
\begin{equation}\label{Esf}
\begin{aligned}
&\quad\E\Bigg[\sum_{t=0}^{T-1}\nabla f(\theta_t)^\top \eta_{t-1}r_{t}\nabla f(\theta_t)\Bigg]\\
&\leq \E\Big[f(\theta_0)-f(\theta_{T})\Big] +  (\eta G^{2}_\infty/\sqrt{a})\E\Big[\|r_{0}\|_{1,1}-\|r_{T}\|_{1,1}\Big]\\
&\quad+ (2\eta/\sqrt{a}+\mu\eta)\E\Bigg[\sum_{t=0}^{T-1}g_{t}^\top r_{t+1}g_{t}\Bigg] + \mu\eta\E\Bigg[\sum_{t=0}^{T-1}m_{t}^\top r_{t+1}m_{t}\Bigg]\\
&\quad+(2\eta/\sqrt{a})\E\Bigg[\sum_{t=0}^{T-1}|(\nabla f(\theta_t)-g_t)^\top r_{t+1}g_{t}|\Bigg]
+2\mu\eta\E\Bigg[\sum_{t=0}^{T-1}|(g_t-\nabla f(\theta_t))^\top r_{t+1}m_{t}|\Bigg]\\
&\quad+\frac{L}{2}\E\Bigg[\sum_{t=0}^{T-1}\|\theta_{t+1}-\theta_{t}\|^2_2\Bigg].
\end{aligned}
\end{equation}
Below we bound each term in (\ref{Esf}) separately. First recall $g_t=2v_t\sqrt{f(\theta_t;\xi_t)+c}$, and $f(\theta_t;\xi_t)+c\leq B$, we have
\begin{align}\notag
\sum_{t=0}^{T-1}g_t^\top r_{t+1}g_t 
\leq 4B \sum_{t=0}^{T-1}v_t^\top r_{t+1}v_t = 4BG(T,0)
\leq 2nB\tilde F_0/\eta,
\end{align}
where \eqref{GTmu2} with $\mu=0$ was used. Note that by $r_{t+1,i}\leq r_{t,i}$, we have
\begin{equation*}
\sum_{t=0}^{T-1}m_{t}^\top r_{t+1}m_{t} 
\leq \sum_{t=0}^{T-1}m_{t}^\top r_{t}m_{t} 
= \sum_{t=0}^{T-2}m_{t+1}^\top r_{t+1}m_{t+1} \leq  G(T,\mu)\leq \frac{ n\tilde F_0}{2\eta(1-\mu)^2},    
\end{equation*}
where \eqref{GTmu2} was used.  These two bounds allow us to further get
\begin{align}\notag
\sum_{t=0}^{T-1}\|r_{t+1}g_t\|^2_2
&=\sum_{i=1}^n\sum_{t=0}^{T-1}  r^2_{t+1, i}g_{t, i}^2  \leq \sum_{i=1}^n\sum_{t=0}^{T-1} r_{0,i} r_{t+1, i}g_{t, i}^2 \\
&=\bigg(\sum_{i=1}^n\sum_{t=0}^{T-1} r_{t, i}g_{t, i}^2\bigg)\tilde F_0\label{rg^2}
\leq 2nB\tilde F^2_0/\eta,
\end{align}
and also 
\begin{align}\label{rm^2}
\sum_{t=0}^{T-1}\|r_{t+1}m_t\|^2_2
\leq \bigg(\sum_{i=1}^n\sum_{t=0}^{T-1} r_{t, i}m_{t, i}^2\bigg)\tilde F_0
&\leq \frac{ n\tilde F^2_0}{2\eta(1-\mu)^2}.
\end{align}
For the rest three terms in (\ref{Esf}), we use the Cauchy-Schwarz inequality to get 
\begin{align}\notag
\E\Bigg[\sum_{t=0}^{T-1}|(g_t-\nabla f(\theta_t))^\top r_{t+1}m_{t}|\Bigg]
&\leq \E\Bigg[\sum_{t=0}^{T-1}\|\nabla f(\theta_t)-g_t\|_2 \|r_{t+1}m_{t}\|_2\Bigg]\\\notag
&\leq \E\Bigg[\bigg(\sum_{t=0}^{T-1}\|\nabla f(\theta_t)-g_t\|^2_2\bigg)^{1/2}\bigg(\sum_{t=0}^{T-1}\|r_{t+1}m_{t}\|^2_2\bigg)^{1/2}\Bigg]\\\notag
&\leq \Bigg(\E\Bigg[\sum_{t=0}^{T-1}\|\nabla f(\theta_t)-g_t\|^2_2\Bigg]\Bigg)^{1/2}\Bigg(\E\Bigg[\sum_{t=0}^{T-1}\| r_{t+1}m_{t}\|^2_2\Bigg]\Bigg)^{1/2}\\\label{vm}
&\leq \sigma_g\sqrt{nT/2\eta}F(\theta_0) /(1-\mu),
\end{align}
where (\ref{rm^2}) and the bounded variance assumption were used.  We replace $m_t$ in (\ref{vm}) by $g_t$ and use (\ref{rg^2}) to get
\begin{align}\notag 
\E\Bigg[\sum_{t=1}^{T}|(\nabla f(\theta_t)-g_t)^\top r_{t+1}g_{t}|\Bigg] 
&\leq \Bigg(\E\Bigg[\sum_{t=0}^{T-1}\|\nabla f(\theta_t)-g_t\|^2_2\Bigg]\Bigg)^{1/2}\Bigg(\E\Bigg[\sum_{t=0}^{T-1}\| r_{t+1}g_{t}\|^2_2\Bigg]\Bigg)^{1/2}\\ \label{vg}
& \leq \sigma_g\sqrt{2BnT/\eta}F(\theta_0).
\end{align}
By (\ref{rev1+}), the last term in (\ref{Esf}) is bounded above by 
\begin{equation}\label{dtheta^2}
\frac{L}{2}\E\left[\sum_{t=0}^\infty\|\theta_{t+1}-\theta_t\|^2 \right] \leq  \frac{L\eta n}{(1-\mu)^2}F^2(\theta_0).
\end{equation}

Substituting (\ref{vg}), (\ref{vm}), (\ref{dtheta^2}) into (\ref{Esf}),  using  Lemma \ref{lem3},  $\E[\|r_0\|_{1, 1}]\leq nF(\theta_0)$, 
we get
\begin{equation}\label{rgb}
\begin{aligned}
\E\Bigg[\sum_{t=0}^{T-1}&\nabla f(\theta_t)^\top \eta_{t-1}r_{t}\nabla f(\theta_t)\Bigg]
\leq (f(\theta_0)-f(\theta^*))+\eta n(G^2_\infty/\sqrt{a})F(\theta_0)\\
&+ \Big(4B/\sqrt{a}+2\mu B+\mu/(2(1-\mu)^2)\Big)n F(\theta_0)\\
& + (2\sqrt{2B/a}+\sqrt{2}\mu/(1-\mu))\sigma_g\sqrt{\eta n T}F(\theta_0) +L\eta nF^2(\theta_0)/(1-\mu)^2.
\end{aligned}    
\end{equation}
Note that the left hand side is bounded from below by 
$$
\eta B^{-1/2} \E\Bigg[  \min_ir_{T,i} \sum_{t=0}^{T-1}\|\nabla f(\theta_t)\|^2_2\Bigg],
$$ 
where we used $\eta_t\geq \eta/B^{1/2}$. 
Thus we have 
\begin{align*}
\E\Bigg[  \min_ir_{T,i} \sum_{t=0}^{T-1}\|\nabla f(\theta_t)\|^2_2\Bigg]
\leq \frac{C_1+C_2n+C_3\sigma_g \sqrt{ nT}}{\eta },
\end{align*}
where
\begin{align*}
C_1 &= (f(\theta_0)-f(\theta^*))B^{1/2},\\
C_2 &=  \Big(\eta G^2_\infty /\sqrt{a} + 4B/\sqrt{a}+2\mu B+\mu/(2(1-\mu)^2)\Big) B^{1/2}\sqrt{f(\theta_0)+c}\\
&\quad+\eta LB^{1/2}(f(\theta_0)+c)/(1-\mu)^2,\\
C_3 &= \big(2\sqrt{B/a}+\mu/(1-\mu)\big)\sqrt{2\eta B}\sqrt{f(\theta_0)+c}.
\end{align*}

\section{Proof of Theorem \ref{thm2r}}\label{pf2r}
Recall that $F(\theta)=\sqrt{f(\theta)+c}$, then for any $x, y\in \{\theta_t\}_{t=0}^T$ we have 
\begin{align*}
\|\nabla F(x)-\nabla F(y)\|
&= \bigg\|\frac{\nabla f(x)}{2F(x)}-\frac{\nabla f(y)}{2F(y)}\bigg\|\\
&= \frac{1}{2}\bigg\|\frac{\nabla f(x)(F(y)-F(x))}{F(x)F(y)} + \frac{\nabla f(x)-\nabla f(y)}{F(y)}\bigg\|\\
&\leq \frac{G_\infty}{2(F(\theta^*))^2}|F(y)-F(x)| + \frac{1}{2F(\theta^*)}\|\nabla f(x)-\nabla f(y)\|.
\end{align*}
One may check that 
$$
|F(y)-F(x)|\leq \frac{G_\infty}{2F(\theta^*)}\|x-y\|.
$$
These together with the $L$-smoothness of $f$ lead to
\begin{equation*}
\|\nabla F(x)-\nabla F(y)\|_2 \leq 
L_F \|x-y\|,
\end{equation*}
where 
$$
L_F=
\frac{1}{2\sqrt{f(\theta^*)+c}} \left( L+ \frac{G^2_\infty}{2(f(\theta^*)+c)}\right). 
$$
This confirms the $L_F$-smoothness of $F$, which yields 
\begin{align*}
 F(\theta_{t+1}) -  F(\theta_t)
& \leq\nabla F(\theta_t)^\top (\theta_{t+1}-\theta_t) +\frac{L_F}{2}\|\theta_{t+1}-\theta_t\|^2 \\
& = (\nabla F(\theta_t)-v_t)^\top (\theta_{t+1}-\theta_t) +v_{t}^\top (\theta_{t+1}-\theta_t) +\frac{L_F}{2}\|\theta_{t+1}-\theta_t\|^2. \\
\end{align*}
Summation of the above over $t$ from $0$ to $T-1$ and taken with the expectation gives 
\begin{equation}\label{ET1}
\E[F(\theta_{T})-F(\theta_{0})]\leq \sum_{i=1}^{3} S_i,    
\end{equation}
where 
\begin{align*}
&S_1=\E\Bigg[\sum_{t=0}^{T-1}v_{t}^\top (\theta_{t+1}-\theta_t)\Bigg],\\
&S_2=\E\Bigg[\sum_{t=0}^{T-1}(\nabla F(\theta_t)-v_t)^\top (\theta_{t+1}-\theta_t)\Bigg],\\
&S_3=\E\Bigg[\sum_{t=0}^{T-1}\frac{L_F}{2}\|\theta_{t+1}-\theta_t\|^2\Bigg].
\end{align*}
Below we bound $S_1, S_2, S_3$ separately. To bound $S_1$, we first note that
\begin{align*}
r_{t+1,i}-r_{t,i}&=-2\eta r_{t+1,i}v^2_{t,i} = v_{t,i}(-2\eta r_{t+1,i}v_{t,i})  \\
&= v_{t,i}\big(-2\eta r_{t+1,i}(m_{t+1,i}-\mu m_{t,i})\big)\\
&= v_{t,i}\big(-2\eta r_{t+1,i}m_{t+1,i}+2\mu\eta r_{t+1,i} m_{t,i}\big)\\
&= v_{t,i}(\theta_{t+1,i}-\theta_i)+2\mu\eta r_{t+1,i} v_{t,i}m_{t,i},
\end{align*}
from which we get
\begin{align*}
S_1
&= \E\Bigg[\sum_{t=0}^{T-1}v_{t}^\top (\theta_{t+1}-\theta_t)\Bigg]\\
&=  \E\Bigg[\sum_{i=1}^{n}\sum_{t=0}^{T-1} r_{t+1,i}-r_{t,i}-2\mu\eta r_{t+1,i} v_{t,i}m_{t,i}\Bigg]\\ 
&=  \sum_{i=1}^{n}\E[r_{T,i}]-n\E[\tilde F_0]-2\mu\eta\E\Bigg[\sum_{i=1}^{n}\sum_{t=0}^{T-1} r_{t+1,i} v_{t,i}m_{t,i}\Bigg],
\end{align*}
where the third equality follows from (\ref{aegdm}c).

For $S_2$, by Cauchy-Schwarz inequality, 
we have
\begin{align*}
S_2
&= \E\Bigg[\sum_{t=0}^{T-1}(\nabla F(\theta_t)-v_t)^\top (\theta_{t+1}-\theta_t)\Bigg]\\
&\leq \E\Bigg[\sum_{t=0}^{T-1}\|\nabla F(\theta_t)-v_t)\|_2 \|\theta_{t+1}-\theta_t)\|_2\Bigg]\\
&\leq \E\Bigg[\bigg(\sum_{t=0}^{T-1}\|\nabla F(\theta_t)-v_t)\|^2_2\bigg)^{1/2}\bigg(\sum_{t=0}^{T-1}\|\theta_{t+1}-\theta_t\|^2_2\bigg)^{1/2}\Bigg]\\
&\leq \Bigg(\E\Bigg[\sum_{t=0}^{T-1}\|\nabla F(\theta_t)-v_t)\|^2_2\Bigg]\Bigg)^{1/2}\Bigg(\E\Bigg[\sum_{t=0}^{T-1}\|\theta_{t+1}-\theta_t\|^2_2\Bigg]\Bigg)^{1/2}\\
&\leq \frac{\sqrt{2}F(\theta_0)}{1-\mu}\sqrt{\eta n T}\sqrt{\frac{G^2_\infty}{8a^3}\sigma^2_f+\frac{1}{2a}\sigma^2_g},
\end{align*}
where the last inequality is by ($v$) in Lemma \ref{lem3} and \eqref{rev1+} in Theorem \ref{prop2}.

For $S_3$, also by \eqref{rev1+} in Theorem \ref{prop2}, we have
\begin{align*}
S_3
= \frac{L_F}{2}\E\Bigg[\sum_{t=0}^{T-1}\|\theta_{t+1}-\theta_t\|^2\Bigg]
\leq \frac{L_F\eta n F^2(\theta_0)}{(1-\mu)^2} .
\end{align*}

With the above bounds on $S_1, S_2, S_3$, (\ref{ET1})  can be rearranged as
\begin{align*}
&\quad F(\theta^*) + 2\mu\eta\E\Bigg[\sum_{i=1}^{n}\sum_{t=0}^{T-1} r_{t+1,i} v_{t,i}m_{t,i}\Bigg]- \frac{L_F\eta n F^2(\theta_0)}{(1-\mu)^2}  -\frac{F(\theta_0)}{1-\mu}\sqrt{\eta n T}\sqrt{\frac{G^2_\infty}{4a^3}\sigma^2_f+\frac{1}{a}\sigma^2_g}\\
&\leq  \sum_{i=1}^{n}\E[r_{T,i}] - n \E[\tilde F_0] + F(\theta_0) \\
&\leq  \Big(\min_i \E[r_{T,i}]+(n-1)\E[\tilde F_0]\Big)- (n-1) \E[\tilde F_0] + \Big(F(\theta_0)-\E[\tilde F_0]\Big)
\\
&\leq \min_i\E[r_{T,i}]+ \E[|F(\theta_0)-\tilde F_0|] \\
& \leq \min_i\E[r_{T,i}]+ \frac{1}{2a^{1/2}}\sigma_f,
\end{align*}
where (iii) in Lemma \ref{lem3} was used. Hence,  
\begin{equation}\label{mu+}
\min_i\E[r_{T,i}]\geq F(\theta^*)+2\mu\eta\E\Bigg[\sum_{i=1}^{n}\sum_{t=0}^{T-1} r_{t+1,i} v_{t,i}m_{t,i}\Bigg]-\eta D_1-\sigma D_3,   
\end{equation}
where $\sigma=\max\{\sigma_f,\sigma_g\}$ and
\begin{align*}
&D_1 = \frac{L_F n F^2(\theta_0)}{(1-\mu)^2},\quad D_3 = \frac{1}{2a^{1/2}} + \frac{F(\theta_0)}{1-\mu}\sqrt{\eta n T}\sqrt{\frac{G^2_\infty}{4a^3}+\frac{1}{a}}.
\end{align*}
The remaining term in (\ref{mu+}) with $\mu=0$ vanishes. For $\mu>0$ we proceed to bound this term as follows:  
\begin{align*}
\sum_{i=1}^{n}\sum_{t=0}^{T-1} r_{t+1,i} v_{t,i}m_{t,i}
&\geq -\frac{1}{2} \sum_{i=1}^{n}\sum_{t=0}^{T-1} r_{t+1,i} v^2_{t,i}-\frac{1}{2}\sum_{i=1}^{n}\sum_{t=0}^{T-1} r_{t+1,i} m^2_{t,i}\\
&\geq -\frac{1}{2}(G(T,0)+G(T-1,\mu))\geq -\bigg(1+\frac{1}{(1-\mu)^2}\bigg)n\tilde F_0/(4\eta), 
\end{align*}
which allows us to obtain 
$$
2\mu\eta\E\Bigg[\sum_{i=1}^{n}\sum_{t=0}^{T-1} r_{t+1,i} v_{t,i}m_{t,i}\Bigg] \geq 
-\frac{\mu}{2}\bigg(1+\frac{1}{(1-\mu)^2}\bigg)n\E[\tilde F_0]\leq -\frac{\mu}{2}\bigg(1+\frac{1}{(1-\mu)^2}\bigg)n F(\theta_0).
$$
Therefore, 
\begin{equation*}
\min_i\E[r_{T,i}]\geq \max\{\sqrt{f(\theta^*)+c}-\eta D_1-\mu D_2-\sigma D_3,0\},   
\end{equation*}
where $\sigma=\max\{\sigma_f,\sigma_g\}$ and
\begin{align*}
&D_1 = \frac{L_F n (f(\theta_0)+c)}{(1-\mu)^2}, \quad D_2 =\frac{1}{2}\bigg(1+\frac{1}{(1-\mu)^2}\bigg)n\sqrt{f(\theta_0)+c},\\
&D_3 = \frac{1}{2a^{1/2}} + \frac{\sqrt{f(\theta_0)+c}}{1-\mu}\sqrt{\eta n T}\sqrt{\frac{G^2_\infty}{4a^3}+\frac{1}{a}}.
\end{align*}

\section{Implementation details of experiments}\label{spep}
We summarize the setup for experiments presented in Section 5 in Table \ref{tab2}, where `BS' and `WD' represent batch size and weight decay employed for each task, respectively. The last four columns are base learning rate that achieves the best final generalization performance for each method in respective tasks.

\begin{table}[ht]
\caption{Training settings in our experiments} 
\centering 
\begin{tabular}{c c c c c c c c} 
\hline\hline
Dataset & Model & BS & WD & SGDM & Adam & AEGD & AEGDM\\ [0.5ex] 
\hline 

MNIST & LeNet-5 & 128 & $1e-4$ & 0.01 & 0.001 & 0.05 & 0.008\\

CIFAR-10 & VGG-16 & 128 & $5e-4$ & 0.03 & 0.0003 & 0.1 & 0.005\\

CIFAR-10 & ResNet-32 & 128 & $1e-4$ & 0.05 & 0.001 & 0.2 & 0.008\\

CIFAR-10 & DenseNet-121 & 64 & $1e-4$ & 0.05 & 0.0005 & 0.2 & 0.02\\

CIFAR-100 & SqueezeNet & 128 & $1e-4$ & 0.3 & 0.003 & 0.2 & 0.02\\

CIFAR-100 & GoogleNet & 128 & $1e-4$ & 0.2 & 0.0003 & 0.2 & 0.03\\[1ex] 
\hline 
\end{tabular}
\label{tab2} 
\end{table}

Figure \ref{fig:cifar10a}, \ref{fig:cifar100a} present the comparison results where the defaults base learning rate for each method:
\begin{itemize}
\item AEGDM: $0.01$.
\item AEGD: $0.1$.
\item SGDM: $0.01$ for VGG-16 (on both CIFAR10 and CIFAR100), $0.1$ for other tasks.
\item AdaBelief, AdaBound, RAdam, Yogi, Adam: $0.001$.
\end{itemize}
is used in all tasks. Also different from the setting reported in Table \ref{tab2}, here we set batch size as $128$ and weight decay as $5\times 10^{-4}$ in all tasks. It can be seen that AEGDM and AEGD generalize better than all other methods, and AEGD display smaller oscillation / faster convergence than AEGD.

The experiments were coded in PyTorch and conducted using job scheduling on Intel E5-2640 v3 CPU (two 2.6 GHz 8-Core) with 128 GB Memory per Node.

\begin{figure}[ht]
\begin{subfigure}[b]{0.33\linewidth}
\centering
\includegraphics[width=1\linewidth]{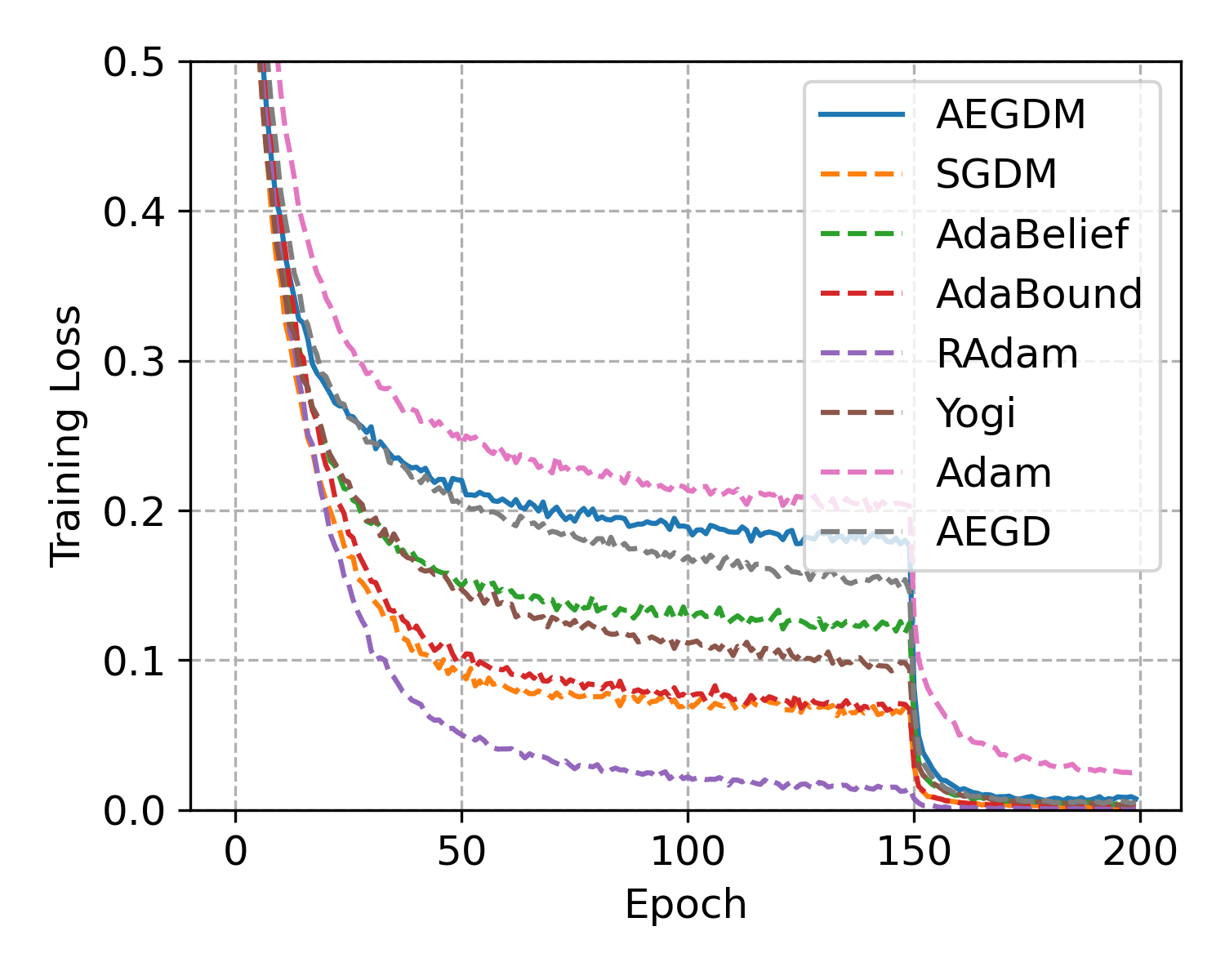}
\caption{VGG-16, training loss}
\end{subfigure}%
\begin{subfigure}[b]{0.33\linewidth}
\centering
\includegraphics[width=1\linewidth]{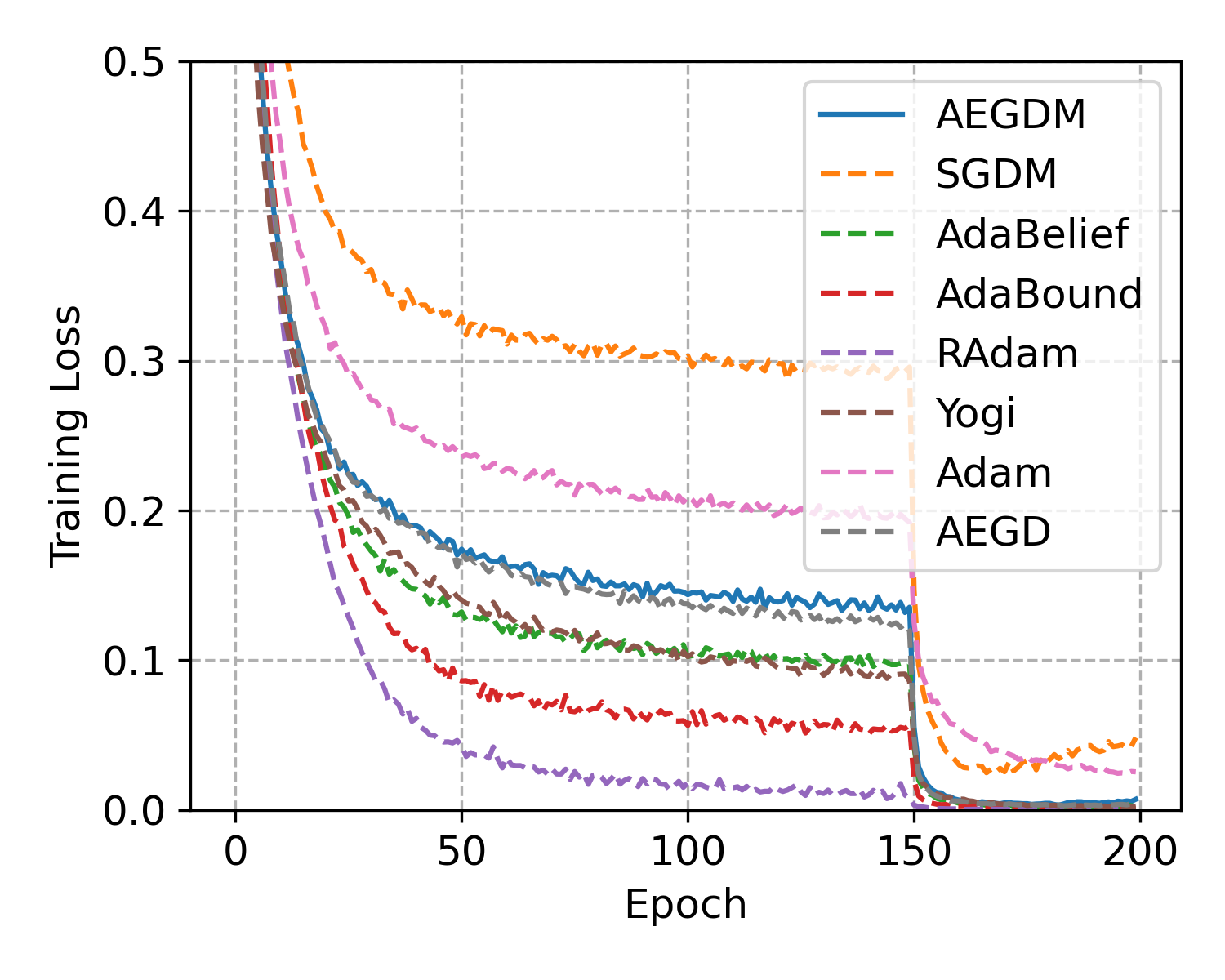}
\caption{ResNet-34, training loss}
\end{subfigure}%
\begin{subfigure}[b]{0.33\linewidth}
\centering
\includegraphics[width=1\linewidth]{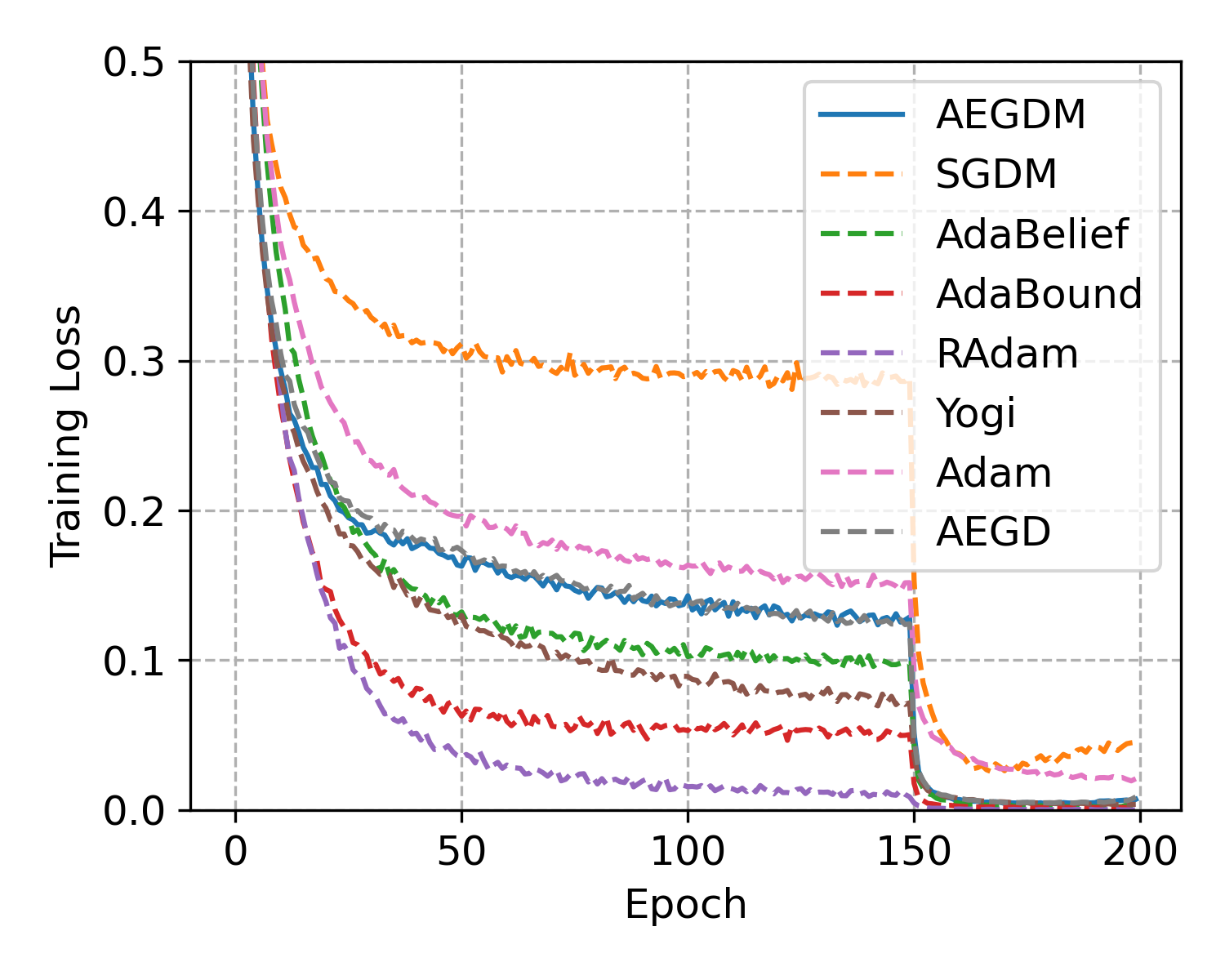}
\caption{DenseNet-121, training loss}
\end{subfigure}%
\newline
\begin{subfigure}[b]{0.33\linewidth}
\centering
\includegraphics[width=1\linewidth]{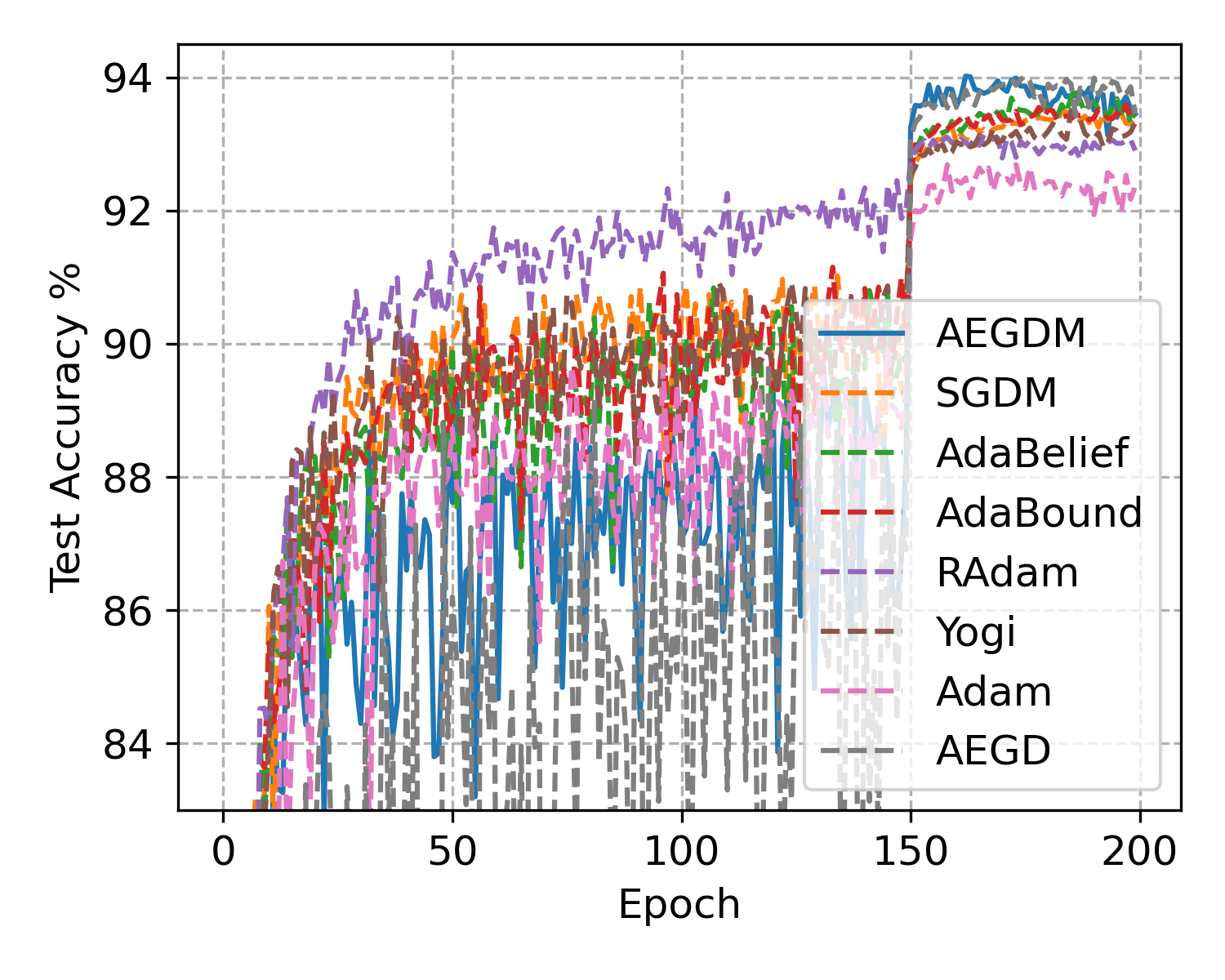}
\caption{VGG-16, test accuracy}
\end{subfigure}%
\begin{subfigure}[b]{0.33\linewidth}
\centering
\includegraphics[width=1\linewidth]{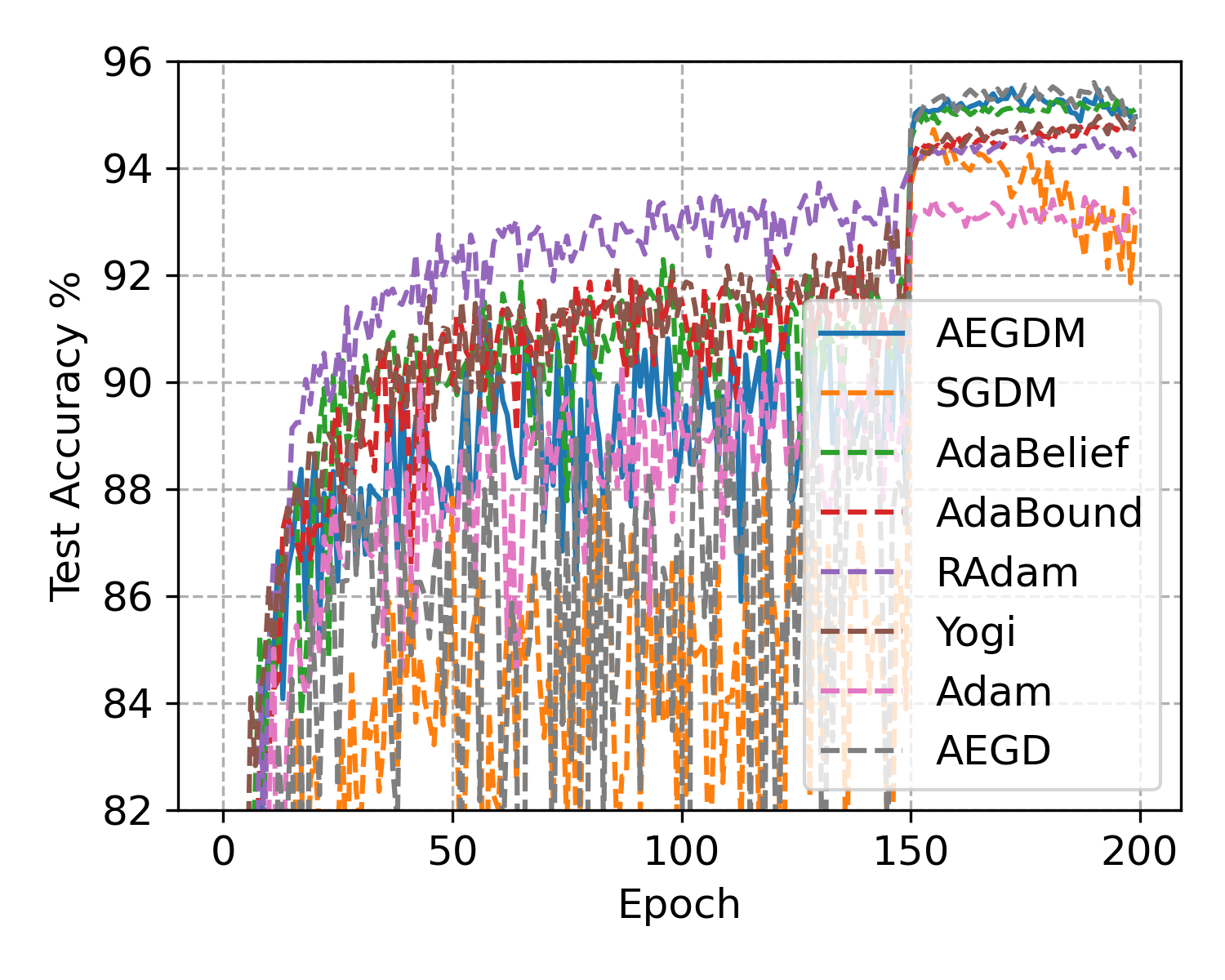}
\caption{ResNet-34, test accuracy}
\end{subfigure}%
\begin{subfigure}[b]{0.33\linewidth}
\centering
\includegraphics[width=1\linewidth]{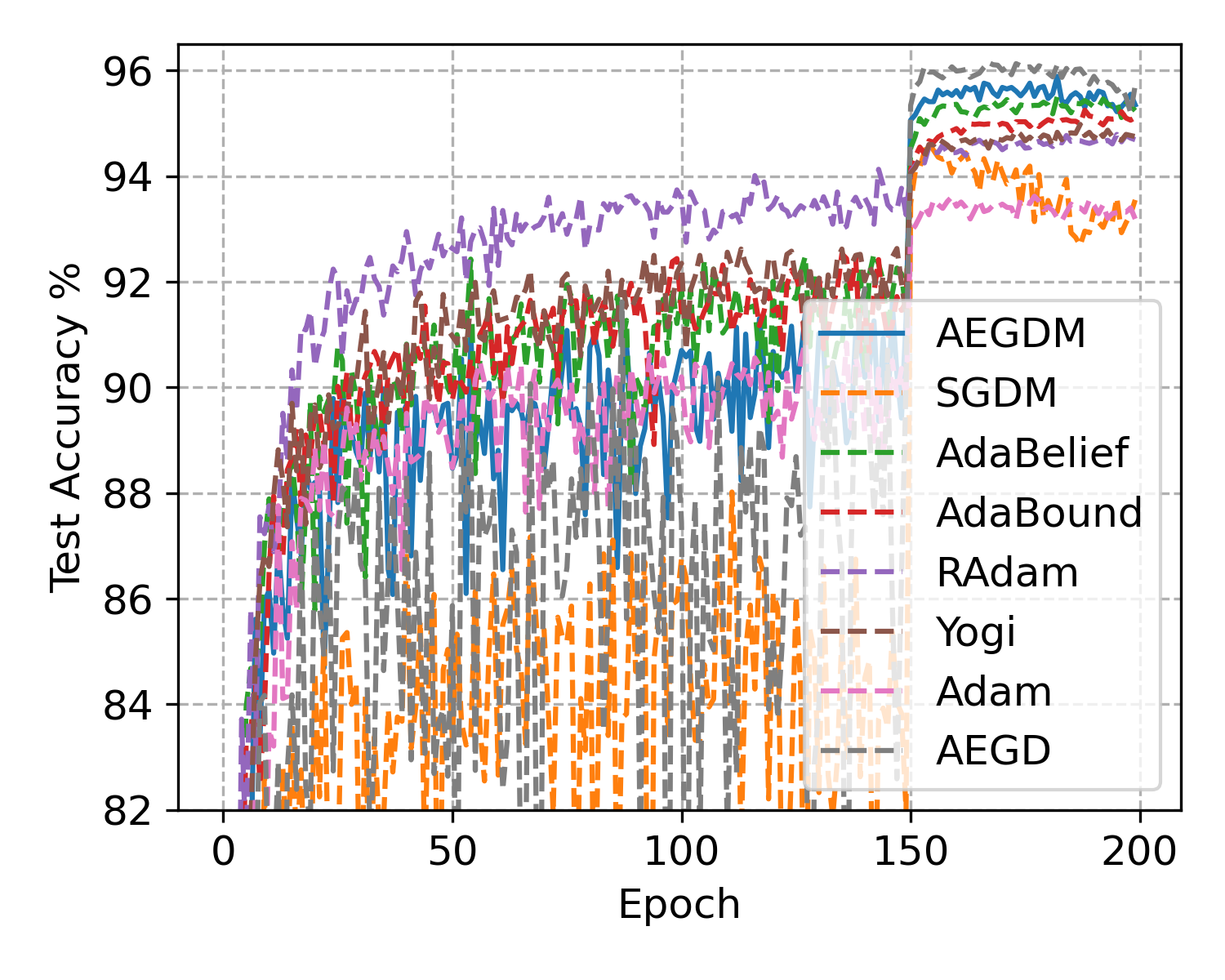}
\caption{DenseNet-121, test accuracy}
\end{subfigure}%
\captionsetup{format=hang}
\caption{Test accuracy for VGG-16, ResNet-34 and DenseNet-121 on CIFAR-10}
\label{fig:cifar10a}
\end{figure}

\begin{figure}[ht]
\begin{subfigure}[b]{0.33\linewidth}
\centering
\includegraphics[width=1\linewidth]{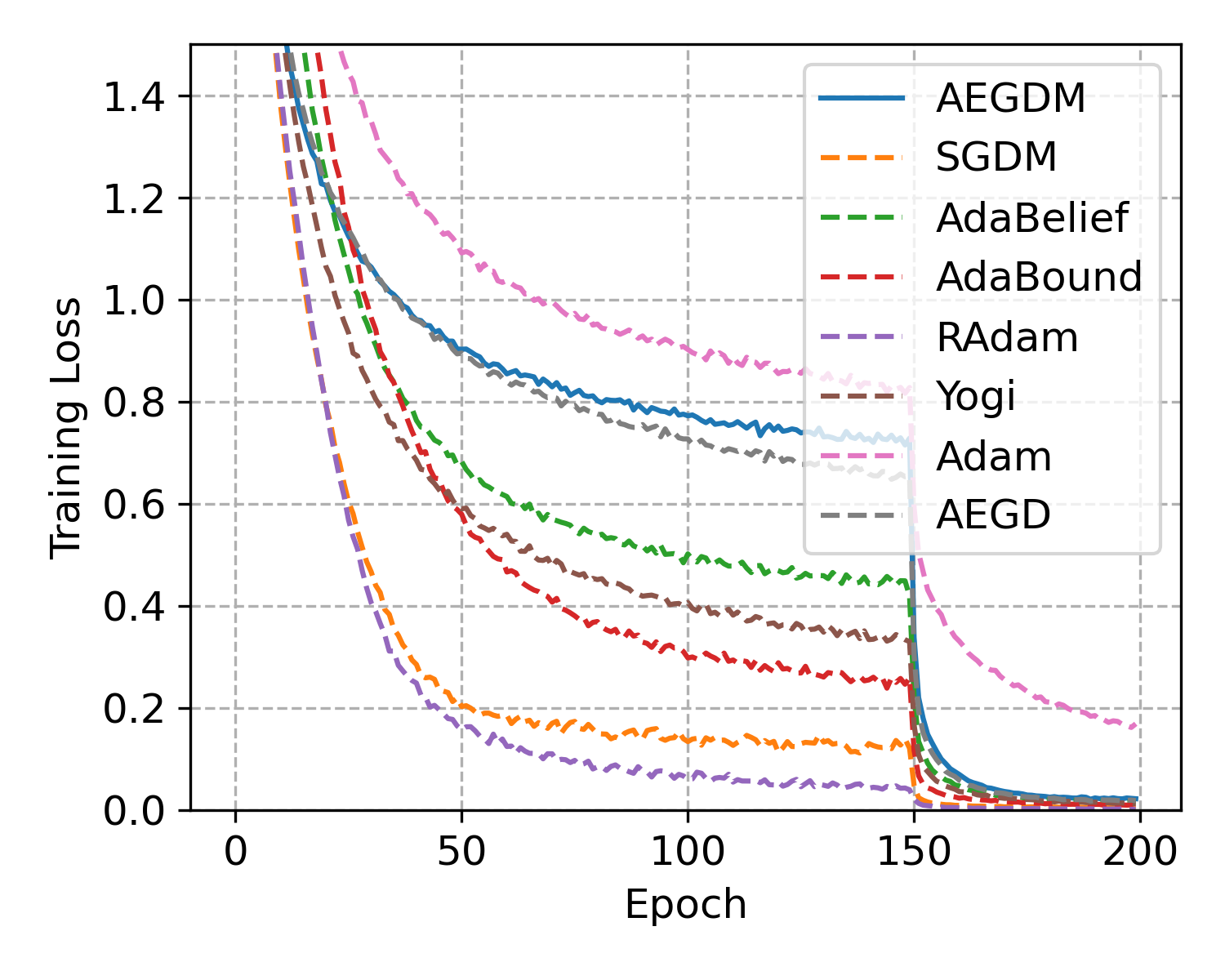}
\caption{VGG-16, training loss}
\end{subfigure}%
\begin{subfigure}[b]{0.33\linewidth}
\centering
\includegraphics[width=1\linewidth]{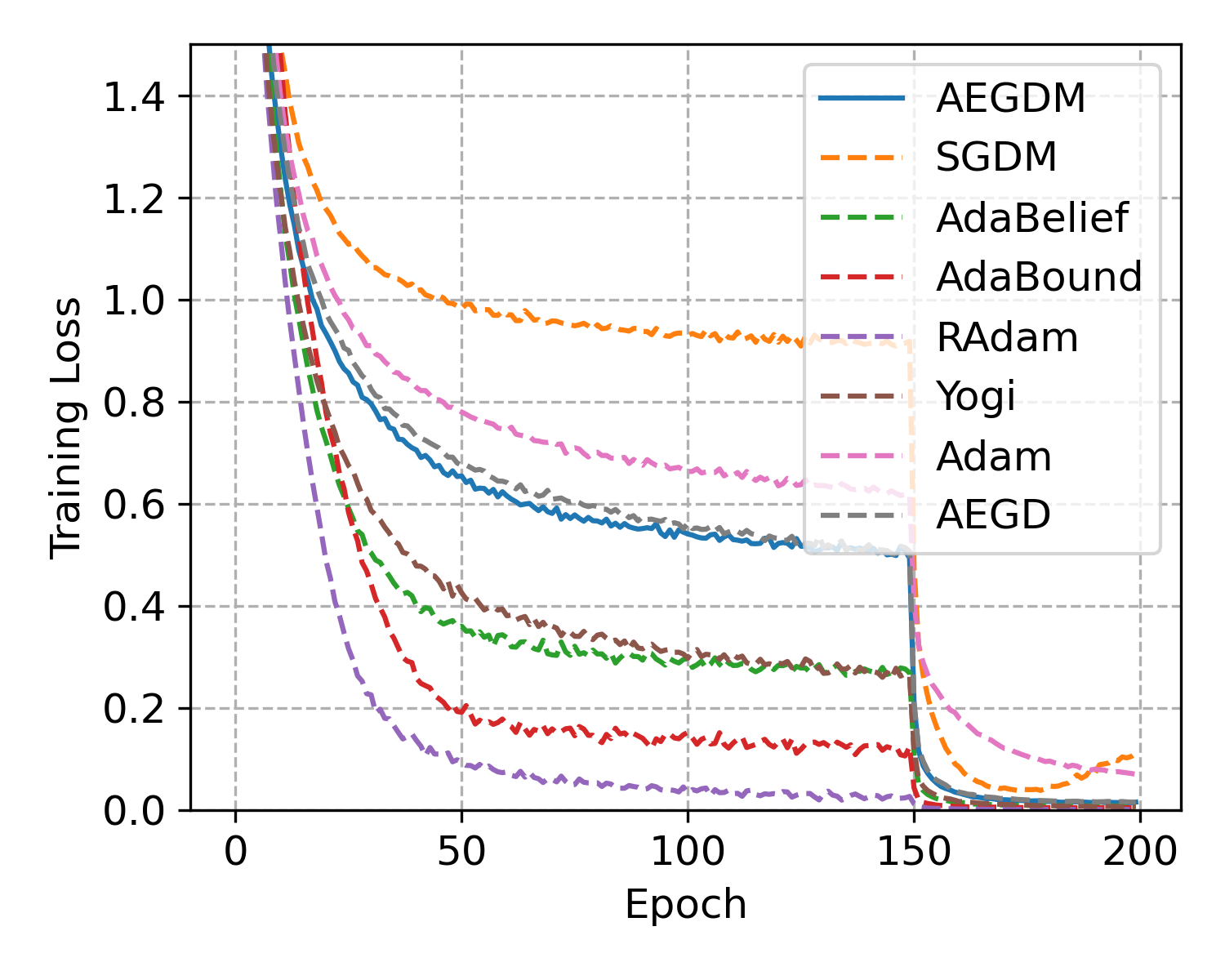}
\caption{ResNet-34, training loss}
\end{subfigure}%
\begin{subfigure}[b]{0.33\linewidth}
\centering
\includegraphics[width=1\linewidth]{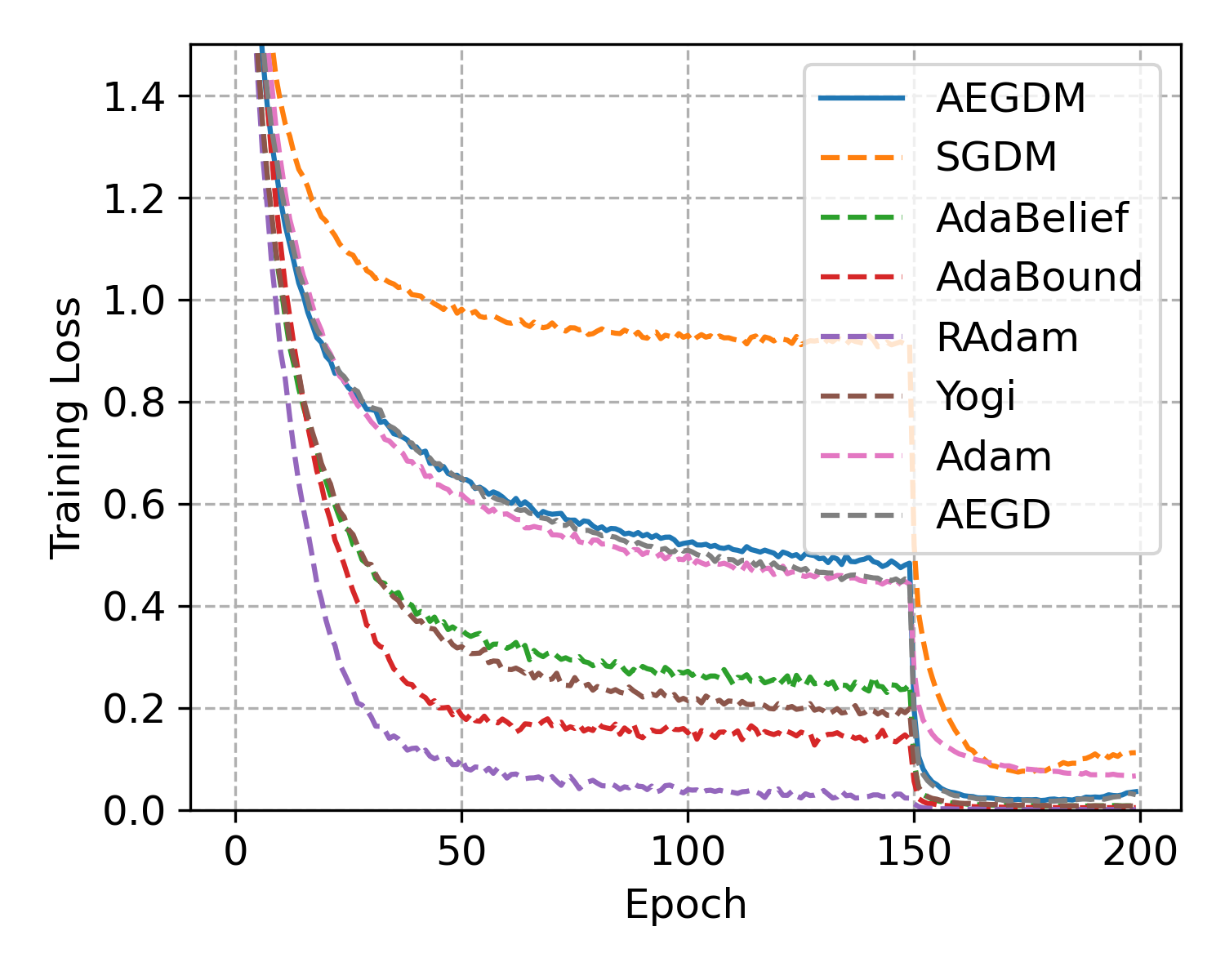}
\caption{DenseNet-121, training loss}
\end{subfigure}%
\newline
\begin{subfigure}[b]{0.33\linewidth}
\centering
\includegraphics[width=1\linewidth]{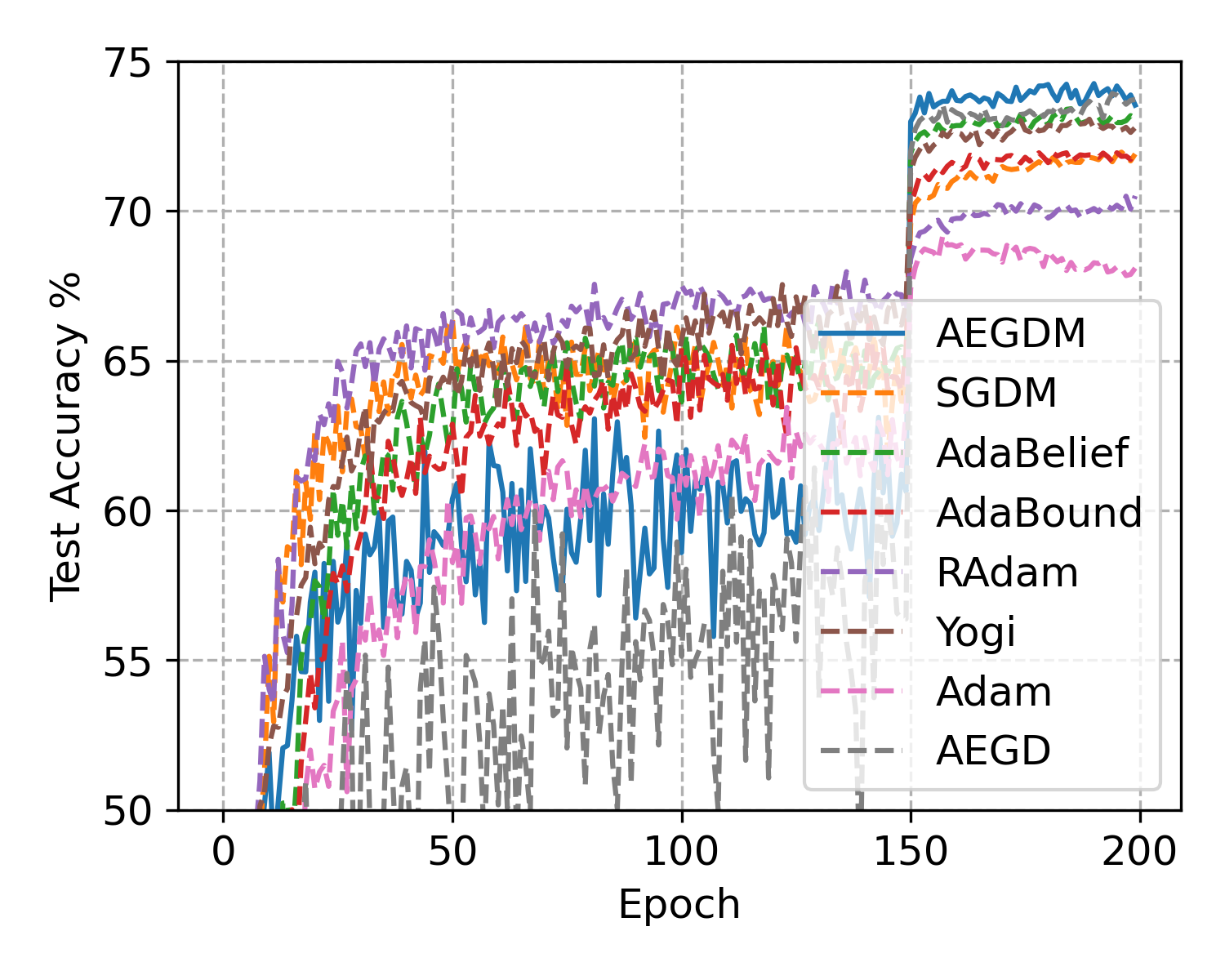}
\caption{VGG-16, test accuracy}
\end{subfigure}%
\begin{subfigure}[b]{0.33\linewidth}
\centering
\includegraphics[width=1\linewidth]{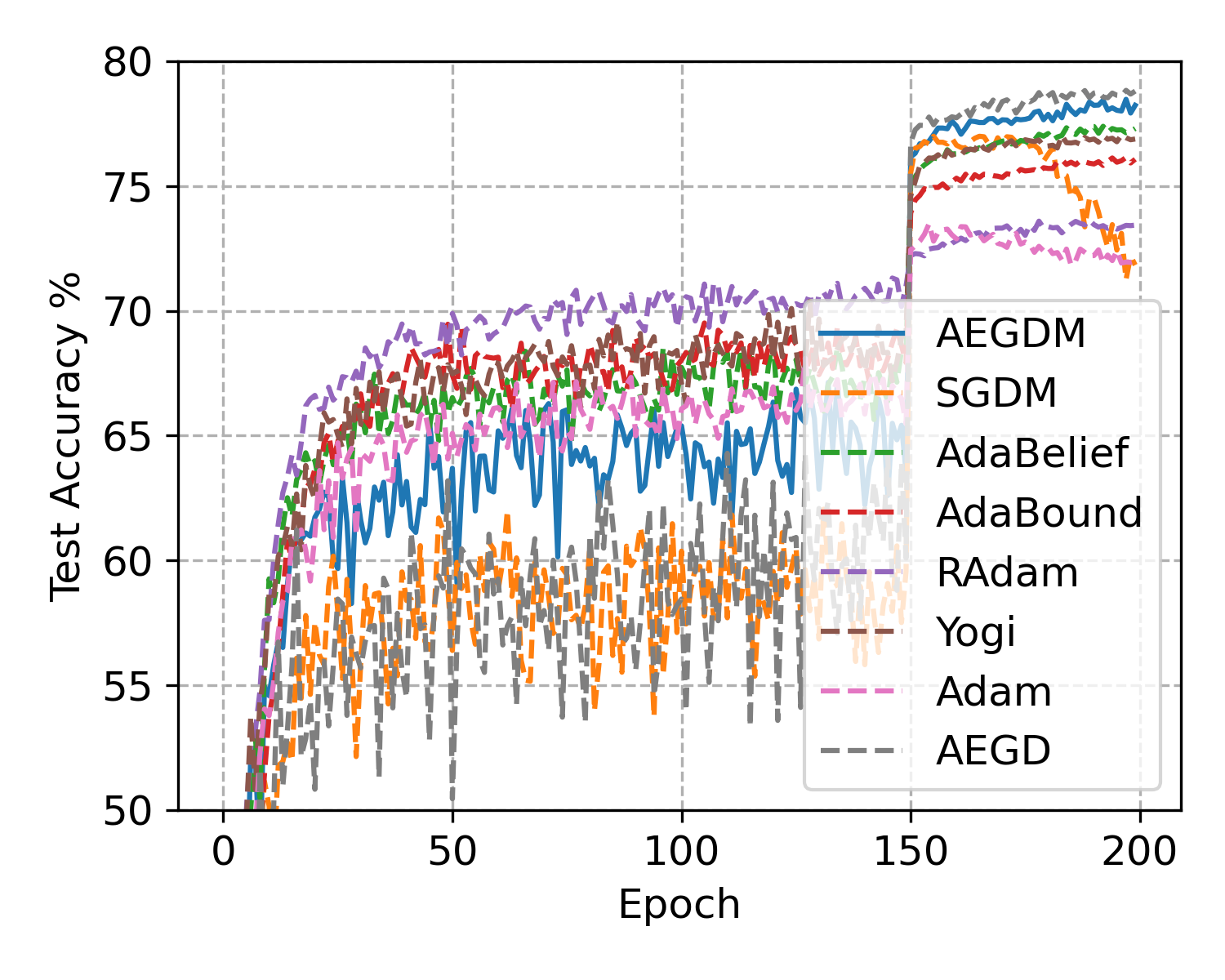}
\caption{ResNet-34, test accuracy}
\end{subfigure}%
\begin{subfigure}[b]{0.33\linewidth}
\centering
\includegraphics[width=1\linewidth]{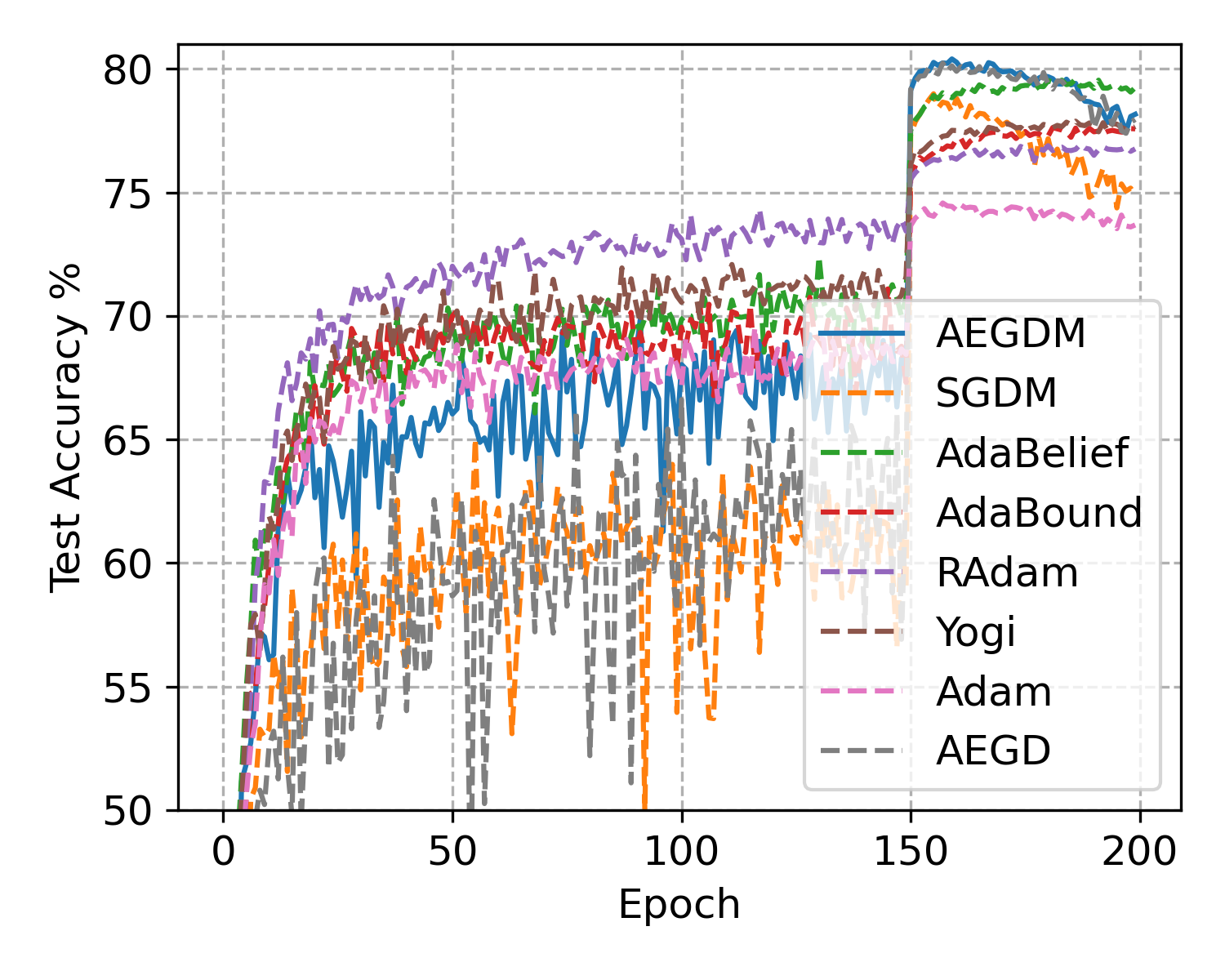}
\caption{DenseNet-121, test accuracy}
\end{subfigure}%
\captionsetup{format=hang}
\caption{Test accuracy for VGG-16, ResNet-34 and DenseNet-121 on CIFAR-100}
\label{fig:cifar100a}
\end{figure}

\begin{figure}[ht]
\centering
\includegraphics[width=0.35\linewidth]{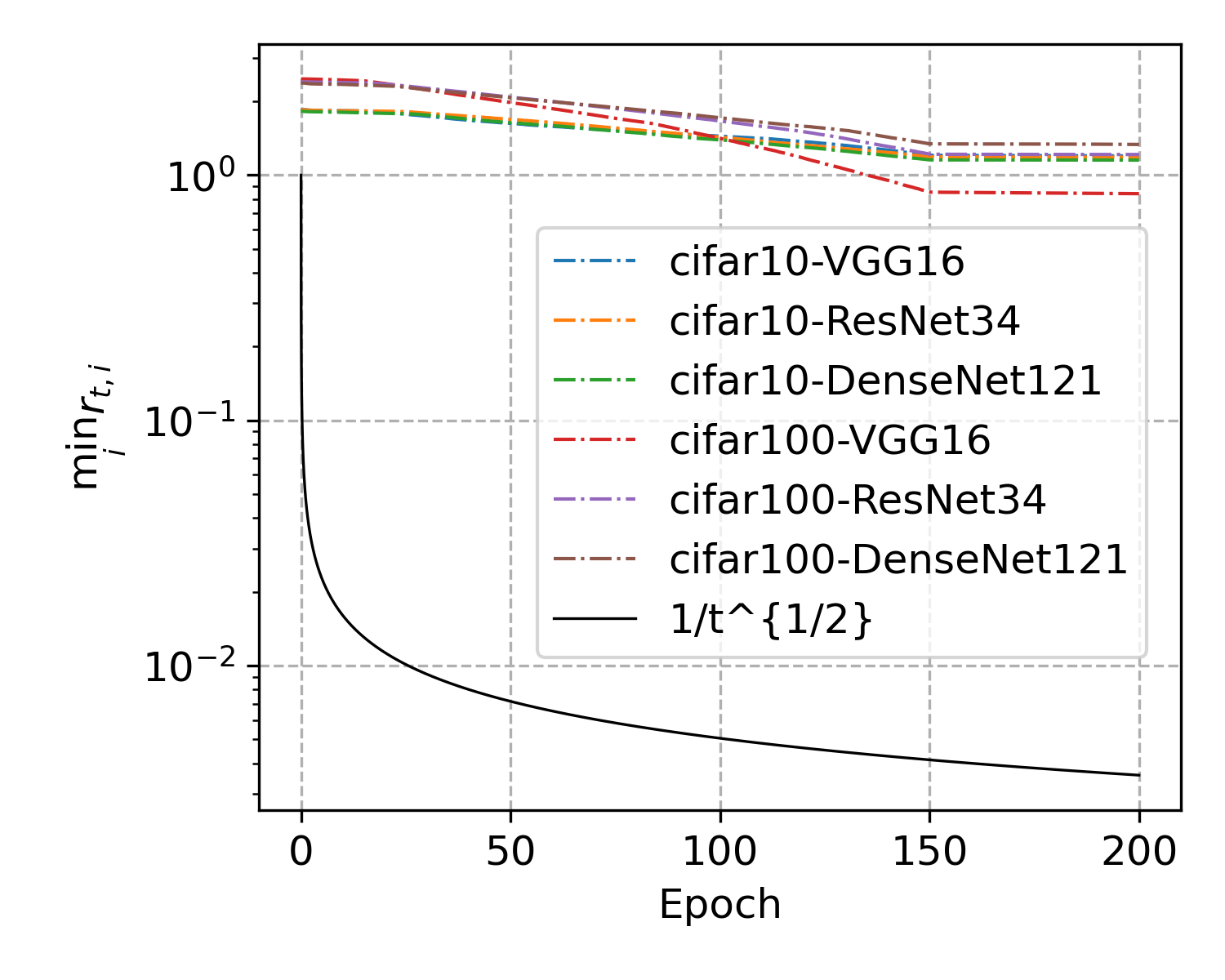}
\caption{$\min_i r_{t,i}$ of AEGDM with default base learning rate $0.01$ in neural network tasks.}
\label{fig:minr}
\end{figure}%

\section{A comparison of some  gradient-based methods} \label{A-framewrok} 
How does AEGDM compare with AEGD, SGD, SGDM, Adam and other adaptive methods?
We apply a more generic formulation, so that these methods will all take the following form 
\begin{equation}\label{afw}
\begin{aligned}
\theta_{t+1}=\theta_t-\eta A_t^{-1}m_{t+1}.
\end{aligned}   
\end{equation}
Here $m_{t+1}$ depends on $\{v_j\}_{j=0}^t$, historical search direction, and $A^{-1}_t$ is a diagonal matrix.
The diagonal form of $A^{-1}_t$ allows for different effective learning rates for different coordinates. The use of inverse $A_t^{-1}$ here is to be consistent with the form of the natural gradient method \cite{A98}, $\theta \leftarrow \theta -\eta A^{-1}\nabla f(\theta)$, in which $A$ is typically a positive definite matrix, playing the role of a metric matrix for certain Riemannian manifold.   
Table \ref{tab} is a comparison in terms of different choices of $A_t^{-1}$ and/or $m_{t+1}$.   

\begin{table}[ht]
\caption{Comparison of some optimization algorithms} 
\centering 
\begin{tabular}{c c c c} 
\hline\hline
  & $v_t$  & $m_{t+1}$ & $A^{-1}_{t}$ \\ [0.5ex] 
\hline 
SGD & $\nabla f_{t}(\theta_t)$ & $v_t$ & $\mathbb{I}$\\

SGDM & $\nabla f_{t}(\theta_t)$ & $\sum_{j=0}^{t}\mu^{t-j}v_j$ & $\mathbb{I}$\\

RMSprop & $\nabla f_{t}(\theta_t)$ & $v_t$ & ${\rm diag}\Big[\Big((1-\beta_2)\sum_{j=0}^{t}\beta^{t-j}_2v_j\odot v_j \Big)^{-\frac{1}{2}}\Big]$\\

ADAM & $\nabla f_{t}(\theta_t)$ & $(1-\beta_1)\sum_{j=0}^{t}\beta^{t-j}_1v_j$ & ${\rm diag}\Big[\Big((1-\beta_2)\sum_{j=0}^{t}\beta^{t-j}_2v_j\odot v_j \Big)^{-\frac{1}{2}}\Big]$ \\[1ex]

AEGD & $\frac{\nabla f_{t}(\theta_t)}{2\sqrt{f_{t}(\theta_t)+c}}$ & $v_t$ & ${\rm diag}\Big[r_0\odot \Big(\prod_{j=0}^{t}(1+2\eta v_j\odot v_j )\Big)^{-1}\Big]$\\[1ex] 
AEGDM & $\frac{\nabla f_{t}(\theta_t)}{2\sqrt{f_{t}(\theta_t)+c}}$ & $\sum_{j=0}^{t}\mu^{t-j}v_j$ & ${\rm diag}\Big[r_0\odot \Big(\prod_{j=0}^{t}(1+2\eta v_j\odot v_j )\Big)^{-1}\Big]$\\[1ex]
\hline 
\end{tabular}
\label{tab} 
\end{table}

Note that in Table \ref{tab}, the negative power function is understood as an element-wise operation. For the hyper-parameters,  $\mu\in (0, 1)$ for SGDM and AEGDM; $r_0=\sqrt{f_{0}(\theta_0)+c}{\bf 1}$ for both AEGD and AEGDM;
$\beta_1 \in (0, 1)$ and $\beta_2 \in (0, 1)$
for Adam and RMSprop, subject to possible unbiased corrections \cite{KB17,TH12}.   


\end{document}